\documentclass[12pt]{amsart}
\usepackage{amsmath,amscd,amssymb,amsfonts,latexsym}
\usepackage[all]{xy}
\usepackage{ulem}

\usepackage[latin1]{inputenc}

\usepackage[T1]{fontenc}

\usepackage{eucal} 
\usepackage{cmmib57} 
\usepackage{amsbsy}
\usepackage{amsthm,alltt}

\usepackage{graphicx,enumerate}

\usepackage{geometry}
\newtheorem{thm}{Theorem}[section]
\newtheorem{prop}[thm]{Proposition}
\newtheorem{lem}[thm]{Lemma}
\newtheorem{cor}[thm]{Corollary}
\theoremstyle{definition}
\newtheorem{definition}[thm]{Definition}
\newtheorem{example}[thm]{Example}
\theoremstyle{remark}
\newtheorem{remark}[thm]{Remark}
\numberwithin{equation}{section}

\newcommand{\CP}{\mathbb{CP}^{n+1}}
\newcommand{\CN}{\mathbb{C}^{n+1}}
\newcommand{\U}{\mathcal{U}}
\newcommand{\C}{\mathbb{C}}
\newcommand{\Z}{\mathbb{Z}}
\newcommand{\Q}{\mathbb{Q}}
\newcommand{\F}{\mathcal{F}}
\newcommand{\K}{\mathcal{L}}
\newcommand{\R}{\mathcal{R}^{\bullet}}
\newcommand{\M}{\mathcal{M}}
\newcommand{\ST}{\mathcal{S}}
\newcommand{\V}{\mathcal{V}}
\newcommand{\G}{\mathcal{G}}
\newcommand{\MHM}{\mathrm{MHM}}

\title{Nearby Cycles and Alexander Modules of Hypersurface Complements}
\author{Yongqiang Liu}
\address{Yongqiang Liu: School of Mathematical Sciences, University of Science and Technology of China, No.96,  JinZhai Road Baohe District, Hefei, Anhui, 230026, P.R.China}
\email {liuyq@mail.ustc.edu.cn}
\date{\today}


\subjclass[2010]{}

\begin{document}

\begin{abstract}  Let $f:\CN \rightarrow \C $ be a polynomial map, which is transversal at infinity.   Using Sabbah's specialization complex, we give a new description of the Alexander modules of the hypersurface complement $\CN\setminus f^{-1}(0)$, and obtain a general divisibility result for the associated Alexander polynomials. As a byproduct, we prove a conjecture of Maxim on the decomposition of the Cappell-Shaneson peripheral complex of the hypersurface. Moreover, as an application, we use nearby cycles to recover the mixed Hodge structure on the torsion Alexander modules, as defined by Dimca and Libgober. We also explore the relation between the generic fibre of $f$ and the nearby cycles.
\end{abstract} 

\maketitle

\section{Introduction}
Let $f:\CN \rightarrow \C $ be a polynomial map, and set $F_{0}=f^{-1}(0)$, $\U=\CN \setminus F_{0}$. The study of the topology of $\U$ and $F_{0}$ is a classical object going back to Zariski. Libgober introduced and studied (infinite cyclic) Alexander invariants associated to $\U$. For $F_{0}$ a plane curve (\cite{L1}, \cite{L2}), or a hypersurface with only isolated singularities, including at infinity (\cite{L3}), he obtained a divisibility result, which states that the only nontrivial global  Alexander polynomial divides the product of the local Alexander polynomials associated with each singular point. 
 
More recently, Maxim (\cite{LM2}) provided generalizations of these results to the case of arbitrary singularities, provided that $f$ is transversal at infinity (i.e., the hyperplane at infinity is generic  with respect to the projective completion of $F_0$), by using the intersection homology theory. In particular, he proved a general divisibility result (\cite{LM2}, Theorem 4.2), which restricts the prime factors (i.e. the zeros) of the global Alexander polynomials to those of the local Alexander polynomials of link pairs around the singular strata. Furthermore, Dimca and Libgober (\cite{DL}) showed that there exist canonical mixed Hodge structures on the (torsion) Alexander modules of the hypersurface complement. 

 On the other hand, Dimca and Nemethi provided the connection between the Alexander invariants of  $\U$ and the monodromy representation of $f$ in \cite{DN} (for the case of $f$ transversal at infinity, see \cite{DL}).  

Assume that $f$ is transversal at infinity. The aim of this paper is to realize the Alexander modules by the Sabbah specialization complex, which is closely related to the nearby cycles of $f$. The nearby cycles provide a good way to glue all the local singular information together: the stalks of the nearby cycles are isomorphic to the cohomology groups of the local Milnor fibre. Therefore, this new approach reveals the relation between the Alexander modules and the singularities of the affine hypersurface $F_{0}$, and provides a new general divisibility result. Moreover, as nearby cycle is an important ingredient in the theory of mixed Hodge modules, we can use nearby cycles to obtain a mixed Hodge structure (MHS for short) on the Alexander modules.

In this paper, all homology and cohomology groups will be assumed to have $\Q$ coefficients if the coefficients are not mentioned. To any complex algebraic variety $X$ and any ring $R$, we associate the derived category of bounded cohomologically $R$-constructible complexes of sheaves on $X$, denoted by $D^{b}_{c}(X,R)$.

We say that a polynomial $f:\CN\rightarrow \C$ is {\it transversal at infinity} if $f$ is reduced, and the projective closure $V$ of $F_{0}$ in $\CP$ is transversal in the stratified sense to the hyperplane at infinity $H=\CP-\CN$. 
Consider the infinite cyclic cover $\U^{c}$ of $\U$ corresponding to the kernel of the linking number homomorphism \begin{center}
 $f_{\ast}: \pi_{1}(\U)\rightarrow \pi_{1}(\C^{\ast})=\Z$
\end{center} induced by $f$. Then under the deck group action, $H_{i}(\U^{c},\Q)$ becomes a $\Gamma=\Q[t,t^{-1}]$-module, which is called the {\it $i$-th Alexander module of the hypersurface complement $\U$}.  For $f$ transversal at infinity, Maxim showed that $H_{i}(\U^{c})$ is a torsion $\Gamma$-module for $i\leq n$ (\cite{LM2}), and we denote by $\delta_{i}(t)$ the corresponding Alexander polynomial.

The paper is organized as follows. 

In section 2, we recall the definitions and the main results on Alexander modules of hypersurface complement. In particular, in section 2.4, we introduce the Sabbah specialization complex functor associated to $f$ (see \cite{Bu2}, or Definition \ref{d2.7} below), which is denoted by: \begin{center}
$\psi^{S}_{f}:  D^{b}_{c}(\CN,\Q) \rightarrow D^{b}_{c}(F_{0},\Gamma)$,
\end{center} and we call $\psi^{S}_{f}\Q_{\CN}$ the {\it Sabbah specialization complex}, which is a torsion $\Gamma$-module sheaf complex. Consider the natural forgetful functor \begin{center}
$\mbox{ for}:  D^{b}_{c}(F_{0},\Gamma)\rightarrow D^{b}_{c}(F_{0},\Q)$,
\end{center}  which maps a torsion $\Gamma$-module sheaf complex to its underlying $\Q$-complex. Let $\psi_{f}\Q_{\CN}$ be the Deligne nearby cycle complex associated to $f$. Then it is known that one has a non-canonical isomorphism (\cite{Br}, page 13): \begin{center}
$\mbox{ for} \circ \psi^{S}_{f} (\Q_{\CN}) \cong  \psi_{f}\Q_{\CN} [-1]$.
\end{center}

In section 3, we prove the main result of this paper, in which we realize the Alexander modules by the Sabbah specialization complex.
\begin{thm} \label{t1.1} Assume that the polynomial $f: \CN\rightarrow \C$ is transversal at infinity. With the above notations, there exist isomorphisms of $\Gamma$-modules  
  \begin{center}
  $ H_{i}({\U}^{c}) \cong H^{2n+1-i}_{c}(F_{0},\psi^{S}_{f}{\Q_{\CN}})$ for $i<n$, 
  \end{center}
  and $ H_{n}({\U}^{c})$ is a quotient $\Gamma$-module $H^{n+1}_{c}(F_{0},\psi^{S}_{f}{\Q_{\CN}})$. By using the forgetful functor, it follows that there exist isomorphisms of $\Q$-vector spaces: \begin{center}
  $ H_{i}({\U}^{c}) \cong H^{2n-i}_{c}(F_{0},\psi_{f}{\Q_{\CN}})$ for $i<n$,
  \end{center}
  and $ H_{n}({\U}^{c})$ is a quotient $\Q$-vector space of $H^{n}_{c}(F_{0},\psi_{f}{\Q_{\CN}})$.
\end{thm} 

  \medskip
  
 By using Corollary 1.3 in \cite{DL} and Theorem \ref{t1.1}, we reveal the connection between the generic fibre of $f$ and the Sabbah specialization complex in section 4:
\begin{thm} \label{t1.2}  Assume that the polynomial $f: \CN\rightarrow \C$ is transversal at infinity. Let $F$ be the generic fibre of $f$. Then there exist isomorphisms of $\Gamma$-modules 
\begin{center}
$H_{i}(F)\cong H^{2n+1-i}_{c}(F_{0},\psi_{f}^{S}\Q_{\CN}) $  for $i< n$
\end{center}  and an injective $\Gamma$-module map $H^{n+1}_{c}(F_{0}, \psi^{S}_{f}\Q_{\CN})\rightarrow H_{n}(F)$, where the action of  $\Gamma$ for $H_{i}(F)$ is induced from the monodromy at the origin, denoted by $\mathcal{M} _{0}(f)_{i}$. By using the forgetful functor, it follows that there exist  isomorphisms of $\Q$-vector spaces:  
\begin{center}
$H_{i}(F)\cong H^{2n-i}_{c}(F_{0}, \psi_{f}\Q_{\CN}) $ for $ i< n$,
\end{center}  and $H^{n}_{c}(F_{0},\psi_{f}\Q)$ is a $\Q$-vector subspace of $H_{n}(F)$, all compatible with the respective monodromy actions.
\end{thm}
In particular, this theorem shows that $H^{n+1}_{c}(F_{0}, \psi_{f}^{S}\Q_{\CN})$ gives a better (upper) estimate for $H_{n}(\U^{c})$ than $H_{n}(F)$ with the monodromy $\mathcal{M} _{0}(f)_{n}$. 
  \medskip 
  
In section 5, we obtain the following divisibility result via the Sabbah specialization complex, which is similar to the general divisibility result proved by Maxim (\cite{LM2}, Theorem 4.2). 
\begin{thm} \label{t1.3}  Assume that the polynomial $f: \CN\rightarrow \C$ is transversal at infinity. Fix a Whitney b-regular stratification for $F_{0}$. Then for any integer $i$, with $1 \leq i \leq n$, the prime factors of the global Alexander polynomial $\delta_{i}(t)$ are among the prime factors of $\Delta^{\ST}_{q}(t)$,  in the range $n-i\leq s=\dim\ST \leq n$ and $2n-2s-i\leq q \leq n-s$, where $\Delta^{\ST}_{q}(t)$ is the $q$-th local Alexander polynomial associated to the link pair of an arbitrary point in the stratum $\ST\subseteq F_{0}$.
\end{thm} 
  
Let us compare this divisibility result with the one proved by Maxim (\cite{LM2}, Theorem 4.2). The difference between these two divisibility results is on the choice of the strata. Both ways has their own advantages. If $V$ is not irreducible, then Maxim's result is much better, since we can get a divisibility result for each irreducible component of $V$ and choose the common factors on all the irreducible components. But if $V$ is irreducible, the divisibility result via the Sabbah specialization complex is better, since we only need to choose strata in $F_{0}$ rather than $V$. 

In section 5.2, we give several applications of realizing Alexander modules by the Sabbah specialization complex; in section 5.3, we prove a conjecture of Maxim (see \cite{LM1}, Conjecture 4.1, and let us call it Maxim's Conjecture): when the nearby cycle $\psi_{f}\Q_{\CN}$ splits in the category of (shifted) perverse sheaves, the prime factors of the global Alexander polynomial are only among the top local Alexander polynomials of link pairs around the singular strata.
\begin{thm} \label{t1.4}  Assume that the polynomial $f: \CN\rightarrow \C$ is transversal at infinity. Fix a Whitney b-regular stratification for $F_{0}$. If $\psi_{f}\Q_{\CN}$ splits, then for a fixed integer $i$, with $1 \leq i \leq n$, a prime element $\gamma\in \Gamma$ divides $ \delta_{i}(t) $ only if $\gamma$ divides one of the polynomials $\Delta^{\ST}_{n-s}(t)$, where $\Delta^{\ST}_{n-s}(t)$ is the top local Alexander polynomial associated to an arbitrary point in the singular stratum $\ST$ in $F_{0}$ of complex dimension $s$, with $s\geq n-i$.
\end{thm} 
In this special case, we also give a partial answer to the multiplicities of the roots of the global Alexander polynomials (see Corollary \ref{c5.13}).

\medskip

Section 6 contains the most important application of realizing Alexander module by the nearby cycles: we recover the MHS on the torsion Alexander modules as defined by Dimca and Libgober in \cite{DL}. Since the nearby cycle functor is a functor in the category of mixed Hodge modules (\cite{S}), $H^{2n-i}_{c}(X_{0},\psi_{f}\Q_{\CN})$ naturally carries a MHS. The difficulty of giving $H_{i}(\U^{c})$ ($i\leq n$) a MHS via the nearby cycles arises for $i=n$: $ H_{n}({\U}^{c})$ is just a quotient $\Q$-vector space of $H^{n}_{c}(X_{0},\psi_{f}\Q_{\CN})$. We will solve this problem by constructing a transversal pair in section 6.1:
\begin{thm} \label{t1.5} Assume that the polynomial $f: \CN\rightarrow \C$ is transversal at infinity. Then for $0 \leq i\leq n $, $H_{i}(\U^{c})$ carries the MHS inherited from the nearby cycles.
\end{thm} 
The construction of a transversal pair can be used to show that, for polynomial transversal at infinity, most facts about torsion Alexander modules can be reduced to the homogeneous polynomial case.  Indeed, let $\widetilde{f}$ be the homogenization of $f$ with respect to the hyperplane at infinity $H$ (i.e., $V=\lbrace \widetilde{f}=0\rbrace$), and let $F_{\widetilde{f}}$ be the corresponding Milnor fibre given by $\widetilde{f}=1$. In section 6.2, we use the transversal pair construction to show that\begin{center}
    $H_{i}(\U^{c})\cong  H_{i}(F_{\widetilde{f}})$ for $i\leq n$,
\end{center}   which gives Proposition 4.9 in \cite{LM2} and offers a new way to construct MHS. In fact, there exists a natural map from $\U^{c}$ to $F_{\widetilde{f}}$, which induces a homotopy $(n+1)$-equivalence, thus generalizing similar results obtained by Libgober in the case of isolated singularities (\cite{L3}, Corollary 4.9). 

In section 6.3, we show that the above mentioned MHS coincide (including the MHS defined by Dimca and Libgober), and prove the following proposition, which gives Corollary 1.6 in \cite{DL}.
\begin{prop} \label{p1.6}
  Assume that the polynomial $f: \CN\rightarrow \C$ is transversal at infinity. The MHS defined on $H_{i}(\U^{c})$  ($i\leq n$) in Theorem \ref{t1.5} is canonical: compatible with the action of $\Gamma$, i.e., $t: H_{i}(\U^{c})\rightarrow H_{i}(\U^{c})$ is a MHS morphism.  Moreover, there exist MHS isomorphisms \begin{center}
$H_{i}(F_{h})\cong H_{i}(F_{c}) \cong H^{2n-i}_{c}(F_{0}, \psi_{f}\Q_{\CN}) \cong H_{i}(\U^{c})$  for $i<n$,
\end{center} and three surjective MHS morphisms: \begin{center}
$H_{n}(F_{c})\rightarrow H_{n}(\U^{c})$; $H_{n}(F_{h})\rightarrow H_{n}(\U^{c})$; $H^{n}_{c}(F_{0}, \psi_{f}\Q_{\CN})\rightarrow H_{n}(\U^{c})$,
\end{center} where $h=f_{d}$, the top degree part of $f$ (assume that $f$ has degree $d$), $F_{h}$ is the Milnor fibre given by $h=1$, and $F_{c}=f^{-1}(c)$ is a generic fibre of $f$ with $c\in \C$. This also shows that the MHS on the generic fibre $H_{i}(F_{c})$ is independent of choice of $c$ for $i<n$.
\end{prop} 

\textbf{Note.} The main results proved in this paper are actually true in a more general setting 
(see \cite{DL}).

\textbf{Acknowledgments.} I would like to express my deep gratitude to my adviser, Laurentiu Maxim, for encouragement and advice. I am grateful to Alexandru Dimca, Anatoly Libgober, J\"{o}rg Sch\"{u}rmann, Yun Su and Kaiho Wong for useful discussions, and the Mathematics Department at the University of Wisconsin-Madison for hospitality during the preparation of this work. The author is supported by China Scholarship Council (file No. 201206340046), and this work is part of the author's thesis.

\section{Preliminaries}
\subsection{Alexander modules}

Let $f=f(x_{1},\cdots,x_{n+1}):\CN \rightarrow \C  $ be a reduced degree $d$ polynomial map, and set $F_{0}=f^{-1}(0)$. 
We say that $f$ is {\it transversal at infinity} if the projective closure $V$ of $F_{0}$ in $\CP$ is transversal in the stratified sense to the hyperplane at infinity $H=\CP\setminus \CN=\lbrace x_{0}=0 \rbrace$. If $f$ is transversal at infinity, the affine hypersurface $F_{0}$ is homotopy equivalent to a bouquet of $n$-spheres, i.e.,  
$$F_{0}\sim \vee_{\mu} S^{n},$$ where $\mu$ denotes the number of spheres in the above join (cf. \cite[page 476]{DP}). It is shown in loc.cit. that $\mu$ can be determined topologically as the degree of the gradient map associated to the homogenization $\tilde{f}$ of $f$.

  Let $\U$ be the affine hypersurface complement: $\U=\CN\setminus F_{0} $, and assume that $V$ has $r$ irreducible components. Then $H_{1}(\U,\Z)={\Z}^{r}$ (\cite{D1},(4.1.3),(4.1.4)), generated by meridian loops $ \gamma_{i} $ about each irreducible component $ V_{i}$ of $V$, $1\leq i\leq r $. 
  
   We have a surjective map: $\pi_{1}(\U)\rightarrow \pi_{1}(\C^{\ast})=\Z  $ induced by $f$, and consider the corresponding infinite cyclic cover $\U^{c}$ of $\U$. Then, under the deck group action, every homology group $H_{i}(\U^{c},\Q)$ becomes a $\Gamma:=\Q[t,t^{-1}]$-module.
\begin{definition}  The $\Gamma$-module $ H_{i}({\U}^{c},\Q)$  is called the \textit{$i$-th Alexander module of the hypersurface complement $\U$}. 
\end{definition}
When $H_{i}(\U^{c})$ is torsion as $\Gamma$-module, we denote by $\delta_{i}(t)$ the corresponding Alexander polynomial. If $H_{i}(\U^{c})=0$, the corresponding Alexander polynomial is $1$ by convention. Since $\U$ has the homotopy type of a finite $(n+1)$-dimensional CW complex, $H_{i}(\U^{c})=0$ for $i>n+1$ and $H_{n+1}(\U^{c})$ is a free $\Gamma$-module. Hence $H_{i}(\U^{c})$ is interesting only for $0\leq i\leq n$.   
   
   \begin{thm}(\cite{LM2}, Theorem (3.6), (4.1))  Assume that the polynomial $f: \CN\rightarrow \C$ is transversal at infinity.  Then $H_{i}({\U}^{c}) $ is a finitely generated semi-simple torsion $\Gamma$-module for $ 0 \leq i\leq n$, and the roots of the corresponding Alexander polynomial $\delta_{i}(t)$ are roots of unity of order $d$.
 \end{thm}
      
\begin{remark} 
  Maxim showed (\cite{LM2}) that $H_{0}(\U^{c})=\Gamma/(t-1)$, and $H_{n+1}({\U}^{c}) $ is a free $\Gamma$-module with rank $\vert\chi(\U)\vert$. In fact, $\chi(\U)=(-1)^{n+1}\mu$, with $H_{n+1}({\U}^{c})=\Gamma^{\mu} $.
\end{remark}

 \subsection{Local system}
  Define a local system $\K $ on $\U$ with stalk $\Gamma$ by the composed map: \begin{center}
$ \pi_{1}(\U) \overset {f_{\ast}}{\to} \pi_{1}(\C^{\ast})   \rightarrow   Aut(\Gamma) $,     
\end{center} with the second map being given by  $1_{\Z}\mapsto t$. 
Here $t$ is the automorphism of $\Gamma$ given by multiplication by $t$. Then we have the following $\Gamma$-module isomorphisms (\cite{LM2}, Corollary 3.4):  \begin{center}
 $H^{2n+2-i}_{c}(\U,\K)\cong H_{i}(\U,\K)\cong H_{i}(\U^{c}) $ for all $i$.
\end{center} 

\subsection{The peripheral complex}
 
By choosing a Whitney stratification of $V$, and using the hyperplane at infinity $H$, we obtain a stratification of the pair $(\CP, V \cup H)$.  Then, for any perversity function $\overline{p}$, the intersection homology complex $IC_{\overline{p}}^{\bullet}(\CP,\K)$ is defined by using Deligne's axiomatic construction (see \cite{B}, \cite{GM1} and \cite{GM2}). In this paper, we mainly use the indexing conventions from \cite{GM2}. In particular, we have the following normalization property: $IC_{\overline{p}}^{\bullet}(\CP,\K)|_{\U}\cong \K[2n+2]$.

\begin{thm} \label{t2.4} (\cite{LM2}, Lemma 3.1) Assume that the polynomial $f: \CN\rightarrow \C$ is transversal at infinity. Let $j$ be the inclusion of $\U$ into $\CP$. Then we have the following quasi-isomorphisms in $D^{b}_{c}(\CP,\Gamma)$: \begin{equation} \label{2.1}
IC_{\overline{m}}^{\bullet}(\CP,\K)\cong j_{!}\K[2n+2];
\end{equation} 
 \begin{equation}
 IC_{\overline{l}}^{\bullet}(\CP,\K)\cong Rj_{\ast}\K[2n+2],
 \end{equation}
 where the middle and logarithmic perversities are defined as: $\overline{m}(s)=[(s-1)/2]$ and $\overline{l}(s)=[(s+1)/2]$. (Note that $\overline{m}(s)+\overline{l}(s)=s-1$, i.e., $\overline{m}$ and $\overline{l}$ are superdual perversities in the sense of \cite{CS}.)
  \end{thm}
  With the above assumptions and notations, the Cappell-Shaneson superduality isomorphism
holds, i.e., one has (\cite{CS}, Theorem 3.3):\begin{center}
$IC_{\overline{m}}^{\bullet}(\CP,\K)^{op}\cong \mathcal{D} (IC_{\overline{l}}^{\bullet}(\CP,\K))[2n+2] $,
\end{center} where if $A^{\bullet}$ is a complex of sheaves, $\mathcal{D} A^{\bullet}$ denotes its Verdier dual. Here $A^{op}$ is the $\Gamma$-module obtained from $\Gamma$-module $A$ by composing all module structures with the involution $t\rightarrow t^{-1}$.

\begin{definition}
The {\it peripheral complex} $\R \in D_{c}^{b}(\CP, \Gamma) $ is defined by the distinguished triangle (see \cite{CS}) \begin{equation}
 IC_{\overline{m}}^{\bullet}(\CP,\K)\rightarrow IC_{\overline{l}}^{\bullet}(\CP,\K)\rightarrow \R[2n+2] \overset {[1]}{\to},
  \end{equation} or, by using Theorem \ref{t2.4}, by \begin{equation}
  j_{!}\K \rightarrow Rj_{\ast}\K \rightarrow \R  \overset {[1]}{\to}.
  \end{equation}
  \end{definition}
Then, up to a shift,  $\R$ is a self dual (i.e., $ \R \cong \mathcal{D}\mathcal{R}^{\bullet op}[-2n-1] $), torsion (i.e., the stalks of its cohomology sheaves are torsion modules) perverse sheaf on $\CP$ (see \cite{LM2}, section 3.2). In fact, $\R$ has compact support on $V\cup H$, and $\R\vert _{V\cup H} \cong  (Rj_{\ast} \K)\vert_{V\cup H}$. 
\begin{remark} $\R$ as defined here equals $\R[-2n-2]$  as defined by Cappell and Shaneson, see \cite{CS} or \cite{LM2}. 
\end{remark}

 \subsection{The Sabbah specialization complex} The Sabbah specialization complex  (see \cite{Sa}, and its reformulation in \cite{Bu2}) can be regarded as a generalization of Deligne's nearby cycle complex. For a quick introduction to the theory of nearby cycles, the reader is advised to consult \cite{D2} and \cite{Ma}. 
 
In our notations, consider the following commutative diagram of spaces and maps:\begin{equation} \label{2.5}
\xymatrix{
F_{0} \ar[r]^{i} \ar[d]^{f}   & \CN  \ar[d]^{f}   & \ar[l]_{l} \U \ar[d]^{f}  & \ar[l]_{\pi} \U^{c} \ar[d]^{\widehat{f}} \\
 \lbrace 0  \rbrace   \ar[r]          & \C    & \ar[l] \C^{\ast}                                            
            & \ar[l]_{\widehat{\pi}} \widehat{\C^{\ast}} 
}
\end{equation}
where $\widehat{\pi}$ is the universal covering of the punctured disk $\C^{\ast}$, and the right hand square of the diagram is cartesian.

\begin{definition} \label{d2.7}  The Sabbah specialization complex functor  of $f$ is defined by \begin{center}
$\psi^{S}_{f} =i^{\ast}Rl_{\ast} R\pi_{!} (l \circ \pi)^{\ast} : D^{b}_{c}(\CN,\Q) \rightarrow D^{b}_{c}(F_{0},\Gamma)$, 
\end{center} and we call $\psi^{S}_{f} \Q_{\CN}$ the Sabbah specialization complex.
\end{definition}
\begin{remark}  
 The definition of Sabbah's specialization complex  is slightly different from that of the nearby cycle complex of Deligne, where $R\pi_{!}$ is replaced by $R\pi_{\ast}$. 
\end{remark}

Consider the natural forgetful functor \begin{center}
$\mbox{ for}:  D^{b}_{c}(F_{0},\Gamma)\rightarrow D^{b}_{c}(F_{0},\Q)$,
\end{center}  which maps a torsion $\Gamma$-module sheaf complex to its underlying $\Q$-complex. Let $\psi_{f}\Q_{\CN}$ be the Deligne nearby cycle complex associated to $f$. It is known that one has a non-canonical isomorphism (\cite{Br}, page 13): \begin{center}
$\mbox{ for}\circ \psi^{S}_{f} (\Q_{\CN}) \cong  \psi_{f}\Q_{\CN} [-1]$.
\end{center} 
For short, in the following we write $\Q$ for the constant sheaf $\Q_{\CN}$.

\begin{lem}(\cite{Bu2}, Lemma 3.4 (b)) Let $i$ and $l$ be the inclusions of $F_{0}$ and $\U$ into $\CN$ as in the diagram (\ref{2.5}). Then, for $\K$ as in section 2.2, \begin{center}
$\psi^{S}_{f}\Q\cong i^{\ast} Rl_{\ast}\K$. 
\end{center} Moreover, we have the following distinguished triangle in $D^{b}_{c}(\CN,\Gamma)$: \begin{equation} \label{2.6}
l_{!}\K \rightarrow Rl_{\ast}\K \rightarrow   i_{!}\psi^{S}_{f}\Q \overset {[1]}{\to}.
\end{equation}
\end{lem}

Then we can identify the Sabbah specialization complex with the restriction of the peripheral complex to $F_{0}$:
\begin{lem} \label{l2.9}  We have a quasi-isomorphism in $D^{b}_{c}(F_{0},\Gamma)$: $\R\vert_{F_{0}}\cong \psi^{S}_{f}{\Q}$.  
\end{lem}
\begin{proof}
Let $j$ and $k$ be the inclusion of $\U$ and respectively $\CN$ into $\CP$.  
Then $j=k\circ l$, hence $Rj_{\ast}= Rk_{\ast}\circ Rl_{\ast}$. Note that $k^{\ast}Rk_{\ast}=id$, so
\begin{center}
$\R\vert _{F_{0}} \cong  (Rj_{\ast} \K)\vert_{F_{0}}  \overset{(1)}{\cong} i^{\ast} k^{\ast} Rj_{\ast}\K\cong i^{\ast} k^{\ast} Rk_{\ast}Rl_{\ast}\K 
\cong  i^{\ast} Rl_{\ast} \K  \overset{(2)}{\cong} \psi^{S}_{f}{\Q}$
\end{center}
where (1) follows from the composition of pull back functors and (2) comes from previous lemma.
\end{proof}

\section{Realizing Alexander Modules by the Sabbah Specialization Complex}
In this section, we prove Theorem \ref{t1.1}, in which we realize the Alexander modules by Sabbah specialization complex. This result plays an important role in our follow-up paper \cite{LiM}, in which we obtain several polynomial identities, e.g.,  relating the top Alexander polynomial of the complement $\U$ with the intersection pairing on the boundary manifold of $\U$. 

The following lemma plays a key role in the proof of Theorem \ref{t1.1}, and follows from a base change isomorphism induced by the transversality of $f$ at infinity.
\begin{lem}\label{l3.1} Assume that the polynomial $f: \CN\rightarrow \C$ is transversal at infinity. Let  $l$ be the inclusion of $\U$ into $\CN$.  Then $H_{c}^{i}(\CN, Rl_{\ast}\K)=0$ for $i \neq n+1$, and $H_{c}^{n+1}(\CN, Rl_{\ast}\K)$ is a free $\Gamma$-module. 
\end{lem}
\begin{proof} Consider the following commutative diagram of inclusions:
\begin{center}
$\xymatrix{
\U   \ar[r]^-l \ar[d]^{k^{\prime}}            &   \CP\setminus H=\CN \ar[d]^{k} \\
\CP\setminus V         \ar[r]^{l^{\prime}}       & \CP 
}$
\end{center}
There is a base change isomorphism associated with this diagram due to the fact that $V$ intersects $H$ transversally (\cite{Sc}, Lemma 6.0.5):
\begin{center}
$k_{!}\circ Rl_{\ast}=Rl^{\prime}_{\ast}\circ k^{\prime}_{!}$
\end{center}
By applying this base change isomorphism to $\K$ on $\U$, we have 
\begin{center}
$ H^{i}_{c}(\CN,Rl_{\ast}\K)\cong H^{i}(\CP,k_{!}Rl_{\ast}\K)\cong H^{i}(\CP,Rl^{\prime}_{\ast} k^{\prime}_{!}\K)\cong  H^{i}(\CP\setminus V, k^{\prime}_{!}\K) $.
\end{center} 

The quasi-isomorphism (\ref{2.1}), restricted over $\CP\setminus V$, gives: \begin{center}
$IC_{\overline{m}}(\CP\setminus V,\K)\cong k^{\prime}_{!}\K[2n+2]$.
\end{center}  Then $k^{\prime}_{!}\K[n+1] $ is a $\overline{m}$-perverse sheaf. Since $\CP\setminus V$ is an $(n+1)$-dimensional affine variety, we have the following vanishing results (\cite{Sc}, Corollary 6.0.4): \begin{center}
$H_{c}^{i}(\CN, Rl_{\ast}\K) \cong H^{i}(\CP\setminus V, k^{\prime}_{!}\K)=0 $ for $i\geq{n+2}$.
\end{center} So it remains to compute $H_{c}^{i}(\CN, Rl_{\ast}\K)$ for $i\leq n+1$.

There exist natural split exact sequences (\cite{B}, Lemma 3.4.3)
\begin{center}
$0 \rightarrow \text{Ext}(H^{i+1}_{c}(\CN, Rl_{\ast}\K),\Gamma)\rightarrow H^{-i}(\CN, \mathcal{D}(Rl_{\ast}\K))\rightarrow \text{Hom}(H^{i}_{c}(\CN, Rl_{\ast}\K),\Gamma)\rightarrow 0$
\end{center} where  $\mathcal{D}(Rl_{\ast}\K)\cong l_{!} \K^{op}[2n+2]$ by the Verdier duality.
Here $\K^{op}$ is the dual of local system $\K$, which is obtained by composing the representation of $\K$ with the involution $t\rightarrow t^{-1}$. Then
\begin{center}
 $H_{c}^{i}(\CN, Rl_{\ast}\K) \cong {\rm Torsion} (H^{2n+3-i}(\CN, l_{!}\K^{op}))\oplus {\rm Free} (H^{2n+2-i}(\CN, l_{!}\K^{op}))$.
\end{center} 

The quasi-isomorphism (\ref{2.1}), restricted over $\CN$, and composed with the involution, gives:  $IC_{\overline{m}}(\CN,\K^{op})\cong l_{!}\K^{op}[2n+2]$. Then $l_{!}\K^{op}[n+1] $ is also a $\overline{m}$-perverse sheaf. Since $\CN$ is an (n+1)-dimensional affine variety, we have the similar vanishing results: 
\begin{center}
$H^{j}(\CN, l_{!}\K^{op})=0  $ for $j\geq{n+2}$.
\end{center}
Putting all things together, we get that: $H_{c}^{i}(\CN, Rl_{\ast}\K)=0$ for $i<n+1$, while the $\Gamma$-module $H^{n+1}_{c}(\CN, Rl_{\ast}\K)$ is free.
\end{proof}

\medskip 
\textbf{Proof of Theorem \ref{t1.1}.} By applying the compactly supported hypercohomology  functor to the distinguished triangle $l_{!}\K \rightarrow Rl_{\ast}\K \rightarrow   i_{!}\psi^{S}_{f}\Q \overset {[1]}{\to}$ of (\ref{2.6}), we have the following long exact sequence: \begin{equation} \label{3.1}
 \cdots \rightarrow H^{2n+2-i}_{c}(\CN,l_{!}\K)\rightarrow H^{2n+2-i}_{c}(\CN,Rl_{\ast}\K) \rightarrow  H^{2n+2-i}_{c}(F_{0}, \psi^{S}_{f}\Q) \rightarrow \cdots
 \end{equation}

 Using Lemma \ref{l3.1}, we have that \begin{equation} \label{3.2}
  H_{c}^{2n+2-i}(\CN, Rl_{\ast}\K)=\left\{ \begin{array}{ll}
\Gamma^{\mu}, & i=n+1, \\
0, & \text{otherwise.}\\
\end{array}\right.
  \end{equation} where the isomorphism $H_{c}^{n+1}(\CN,Rl_{\ast}\K)=\Gamma^{\mu}$ is induced from the long exact sequence (\ref{3.1}). In fact, the Sabbah specialization complex is torsion; and \begin{center}
$H_{c}^{2n+2-i}(\CN,l_{!}\K)\cong H_{c}^{2n+2-i}(\U,\K)\cong H_{i}(\U,\K)\cong H_{i}(\U^{c})$,
\end{center} where $H_{i}(\U^{c})$ is torsion for $i< n+1$ and $H_{n+1}(\U^{c})=\Gamma^{\mu}$.

 Therefore we have the following $\Gamma$-module isomorphisms  for $i<n$, \begin{center}
$H_{i}({\U}^{c})\cong H^{2n+1-i}_{c}(F_{0}, \psi^{S}_{f}\Q)$
\end{center} and an exact sequence for $i=n$: \begin{equation} \label{3.3}
\xymatrix{
0  \ar[r] & \Gamma^{\mu}   \ar[r]   & \Gamma^{\mu} \ar[r]  & H^{n+1}_{c}(F_{0},\psi^{S}_{f}\Q)   \ar[r]  & H_{n}({\U}^{c})  \ar[r] & 0 
}
\end{equation}

\begin{remark} By using the Verdier duality on the equality (\ref{3.2}) and comping with the involution, we have: \begin{equation} \label{3.4}
H^{i}(\CN, l_{!}\K)=\left\{ \begin{array}{ll}
\Gamma^{\mu}, & i=n+1, \\
0, & \text{otherwise.}\\
\end{array}\right.
\end{equation}
 Similar results are also true for the corresponding extensions of $\K$ over $\CP\setminus V$.
\end{remark}

\section{Generic Fibre and the Sabbah Specialization Complex} 
It is well known that there exists a (minimal) finite bifurcation set $ B_{f}\subset\C $ such that $f$ is a locally trivial fibration over $\C\setminus{B_{f}} $. If $ c\in{\C\setminus {B_{f}}} $, then $F=f^{-1}(c)$ is called the generic fibre of $f $. Set $T_{b}=f^{-1}(D_{b})$, and $T_{b}^{\ast}= f^{-1}(D_{b}^{\ast})$, where $D_{b}$ is a sufficiently small open disk around $b\in B_{f}$, and $ D_{b}^{\ast}=D_{b}\setminus \lbrace b\rbrace $. 

If $0\notin B_{f}$, then $H_{i}(\U^{c})$ is a free $\Gamma$-module for all $i>0$ (\cite{DN}, 2.12), hence this case is not really interesting from the point of view of Alexander-type invariants. So, in this paper, we always assume $0\in B_{f}$, and set $B_{f}^{\ast}=B_{f}\setminus \lbrace 0\rbrace$.

\subsection{Proof of Theorem 1.2}

Consider the compactly supported cohomology long exact sequence induced by the inclusion of $T_{0}^{\ast}$ into $\U$: \begin{center}
$\cdots \rightarrow H^{2n+2-i}_{c}(T_{0}^{\ast},\K) \rightarrow H^{2n+2-i}_{c}(\U,\K) \rightarrow H^{2n+2-i}_{c}(\U\setminus T_{0}^{\ast},\K) \rightarrow \cdots$
\end{center}
Recall that \begin{center}
$H_{c}^{2n+2-i}(\U,\K)\cong H_{i}(\U,\K)\cong H_{i}(\U^{c})$.
\end{center} Since $f$ is a fibration over $D_{0}^{\ast}$  with fiber $F$, we have the similar isomorphisms:\begin{center}
$H_{c}^{2n+2-i}(T_{0}^{\ast},\K)\cong H_{i}(T_{0}^{\ast},\K)\cong H_{i}(F)$.
\end{center}
Then the above long exact sequence becomes the homology long exact sequence induced by the natural inclusion of $F$ into $\U^{c}$ (\cite{DN}). Theorem  2.10 in \cite{DN} shows that $H^{2n+2-i}_{c}(\U\setminus T_{0}^{\ast},\K)$ is a free $\Gamma$-module with rank $\sum_{b\in B_{f}^{\ast}} \dim H_{i}(T_{b},F)$  for all $i$.

If $f$ is transversal at infinity, Dimca and Libgober showed (\cite{DL}) that, for $b\in B_{f}^{\ast}$, $f^{-1}(b)$ has only isolated singularities, and $T_{b}$ is obtained from $F$ by adding $\mu_{b}$ (n+1)-cells, where $\mu_{b}$ is the sum of Milnor numbers of the singular points in $f^{-1}(b)$. Recall that for $f$ transversal at infinity,  $F_{0}\sim \vee_{\mu} S^{n}$.
It is easy to see that $\mu=\sum_{b\in B_{f}^{\ast}}\mu_{b}$, so  \begin{equation} \label{4.1}
H^{2n+2-i}_{c}(\U\setminus T_{0}^{\ast},\K)=\left\{ \begin{array}{ll}
\Gamma^{\mu}, & i=n+1, \\
0, & \text{otherwise.}\\
\end{array}\right.
\end{equation}

  Consider the compactly supported long exact sequence associated to the distinguish triangle $l_{!}\K \rightarrow Rl_{\ast}\K \rightarrow   i_{!}\psi^{S}_{f}\Q \overset {[1]}{\to}$. We have the following commutative diagram induced by the inclusion of $T_{0}$ into $\CN$: \begin{center}
$\xymatrix{
\ar[r] & H^{2n+1-i}_{c}(F_{0}, \psi^{S}_{f}\Q)  \ar[r]^{(1)}    \ar[d]^{id} & H^{2n+2-i}_{c}(T_{0}, l_{!}\K)  \ar[r]    \ar[d]  &  H^{2n+2-i}_{c}(T_{0}, Rl_{\ast}\K)  \ar[r]    \ar[d]^{(2)}    & \\
 \ar[r] &  H^{2n+1-i}_{c}(F_{0}, \psi^{S}_{f}\Q)   \ar[r]      &         \ar[r]  H^{2n+2-i}_{c}(\CN, l_{!}\K)  \ar[r]          &  H^{2n+2-i}_{c}(\CN, Rl_{\ast}\K)  \ar[r]            & 
}$
\end{center}
The first horizontal long exact sequence in the diagram shows that $H^{2n+2-i}_{c}(T_{0}, Rl_{\ast}\K)$ is a torsion $\Gamma$-module for all $i$. In fact, the Sabbah specialization complex is torsion, and $H_{c}^{2n+2-i}(T_{0},l_{!}\K) \cong H_{i}(F)$ is a torsion $\Gamma$-module. Then it follows from the equality (\ref{3.2}) that the map (2) in the diagram is $0$ for all $i$.

So we have a short exact sequence: \begin{equation} \label{4.2}
 0 \rightarrow   H^{2n+2-i}_{c}(\CN, Rl_{\ast}\K) \rightarrow  H^{2n+2-i}_{c}(\CN\setminus T_{0}, \K)\rightarrow H^{2n+2-i+1}_{c}(T_{0}, Rl_{\ast}\K)\rightarrow 0
 \end{equation}
  Note that $\CN\setminus T_{0}=\U\setminus T_{0}^{\ast}$. Then the equality (\ref{3.2}) and (\ref{4.1}) gives: \begin{center}
$H^{2n+2-i}_{c}(T_{0}, Rl_{\ast}\K)= 0$, for $i\neq n$.
\end{center} 
Thus the map (1) in the diagram is surjective for $i<n$. The torsion $\Gamma$-modules $H^{2n+1-i}_{c}(F_{0}, \psi^{S}_{f}\Q)$ and $H^{2n+2-i}_{c}(T_{0}, l_{!}\K)\cong H_{i}(F)$, are both isomorphic to $H_{i}(\U^{c})$ for $i<n$ (by Theorem \ref{t1.1} in this paper and Corollary 1.3 in \cite{DL}), so the map (1) is in fact an isomorphism for $i<n$.

  Therefore, for $i=n$, we have a short exact sequence:   
 \begin{center}
$0\rightarrow H^{n+1}_{c}(F_{0},\psi^{S}_{f}\Q) \rightarrow  H_{n}(F) \rightarrow H^{n+2}_{c}(T_{0}, Rl_{\ast}\K) \rightarrow 0 
$
\end{center} Then the map $ H^{n+1}_{c}(F_{0}, \psi^{S}_{f}\Q)\rightarrow H_{n}(F)$ is injective, and the proof of Theorem \ref{t1.2} is completed.

\begin{remark} \label{r4.1} Note that $\chi(F_{0},\psi_{f}\Q)=\chi_{c}(F_{0},\psi_{f}\Q)$ (\cite{D2}, Corollary 4.1.23). The difference between $\dim H^{n}_{c}(F_{0}, \psi_{f}\Q)$ and $\dim H_{n}(F)$ is measured  by $\vert \chi(F)-\chi(F_{0},\psi_{f}\Q) \vert$. In particular, if $F_{0}$ has only isolated singular points, then $\chi(F)=\chi(F_{0},\psi_{f}\Q)$ (\cite{D1}, (5.4.4)), which implies that $H_{n}(F)\cong H^{n}_{c}(F_{0},\psi_{f}\Q)$.
\end{remark}

\subsection{The case $\mu =0$} Recall that if $f$ is transversal at infinity, then $F_{0}\sim \vee_{\mu} S^{n}$, and $H_{n+1}(\U^{c})=\Gamma^{\mu}$.  The proof of Theorem \ref{t1.1} and Theorem \ref{t1.2} shows that the number $\mu$ is crucial for the $n$-th Alexander module $H_{n}(\U^{c})$.

Let us write $f=f_{d}+f_{d-1}+\cdots +f_{0}$, where $f_{i}$ is the homogeneous degree $i$ part of $f$. Denote $h=f_{d}$, and $F_{h}$ the corresponding Milnor fibre given by $h=1$.  Maxim showed that the Alexander modules can also be realized by the homology of $F_{h}$ (\cite{LM2}, Proposition 4.9). Hence we have the following result:

\begin{prop} \label{p4.2}  Assume that the polynomial $f: \CN\rightarrow \C$ is transversal at infinity. Then we have the following $\Gamma$-module isomorphisms
\begin{equation} \label{4.3}
H_{i}(F_{h})\cong H_{i}(F)\cong H^{2n+1-i}_{c}(F_{0}, \psi^{S}_{f}\Q)\cong H_{i}(\U^{c}) \text{ for  } i<n,
\end{equation}
 and $H_{n}(\U^{c})$ is a quotient $\Gamma$-module of either of $H_{n}(F_{h})$, $H_{n}(F)$ and $H^{n+1}_{c}(F_{0}, \psi^{S}_{f}\Q)$. The monodromy action on $H_{i}(F)$ is explained in \cite{DL}. Moreover, when $\mu=0$, the isomorphism (\ref{4.3}) holds for all $i$, which is just like the homogeneous polynomial case.
 \end{prop}
 \begin{proof} It remains to prove that when $\mu=0$, the isomorphism (\ref{4.3}) holds for all $i$. Assume that $\mu=0$, then $H_{n+1}(\U^{c})=0$, and all the cohomology groups in the equality (\ref{3.2}) and (\ref{4.1}) vanish. So, the exact sequence (\ref{3.3}) becomes an isomorphism: $H^{n+1}_{c}(F_{0},\psi^{S}_{f}\Q) \cong H_{n}({\U}^{c}) $, and the short exact sequence (\ref{4.2}) gives: $H^{2n+2-i}_{c}(T_{0}, Rl_{\ast}\K)= 0$ for all $i$, hence \begin{center}
$ H^{2n+1-i}_{c}(F_{0}, \psi^{S}_{f}\Q) \cong   H^{2n+2-i}_{c}(T_{0}, l_{!}\K) \cong H_{i}(F) $ for all $i$.  \end{center}
The equality (\ref{3.4}) gives $H^{i}(\CN, l_{!}\K)=0$ for all $i$. By the proof of Proposition 4.9 in \cite{LM2}, it follows that $H_{i}(\U^{c})\cong H_{i}(F_{h})$ for all $i$.
 \end{proof}

Huh provided (\cite{Huh}) a positive lower bound of $\mu$ for isolated singularities case, unless $V$ is a cone (i.e., $f$ is a homogeneous polynomial up to a change of coordinates).  Theorem 1 in \cite{Huh} can be rephrased as follows:
\begin{thm} \label{t4.3} Assume that the polynomial $f: \CN\rightarrow \C$ is transversal at infinity and $F_{0}$ has only isolated singularities. Let $m$ be the multiplicity of $F_{0}$ at one of its singular points $x$. Then $$\mu \geq (m-1)^{n},$$ unless $V$ is a cone with the apex $x$. In particular, $\mu=0$ if and only if, $V$ is a cone.
\end{thm}

\begin{remark}
This theorem is not true for non-isolated singularities case. In fact, consider the threefold in $\mathbb{CP}^{4}$ defined by : \begin{center}
 $V=\lbrace x_{3}^{d-1}x_{0}+ x_{3}^{d-2}x_{4}x_{1}  + x_{4}^{d-1}x_{2}=0 \rbrace$, with $d\geq 3$.
 \end{center} $\mu$ can be determined topologically as the degree of the gradient map of the defining equation of $V$, see \cite{DP}. It is known that $\mu=0$ and $V$ is not a cone (see \cite{Huh}).  
 \end{remark}

We can use Proposition \ref{p4.2} to give a different proof of the last claim in Theorem \ref{t4.3} for the plane curve case $(n=1)$:
\begin{prop} Assume that the degree $d$ polynomial $f: \C^{2}\rightarrow \C$ is transversal at infinity. If $\mu=0$, then $f$ is a homogeneous polynomial up to a change of coordinates.
\end{prop}
\begin{proof} $V\cap H$ is defined by the homogeneous polynomial $h=f_{d}=0$. For $n=1$ and $f$ transversal at infinity, $V\cap H$ is $d$ distinct points, so, without loss of generality, we can assume $h=f_{d}=x_{1}^{d}+x_{2}^{d}$. When $\mu=0$, Proposition \ref{p4.2} gives $H_{1}(F_{h})\cong H_{1}(\U^{c})$ as $\Gamma$-module, so \begin{center}
 $\delta_{1}(t)=(t-1)^{d-1}(\dfrac{t^{d}-1}{t-1})^{d-2}$.
\end{center} The multiplicity of the factor $(t-1)$ in $\delta_{1}(t)$ is $d-1$, so $V$ has $d$ irreducible components (\cite{Oka}, Lemma 21). Since $f$ is a degree $d$ polynomial, this shows that $V$ is a hyperplane arrangement. 

Since $f$ is transversal at infinity, it follows that there can not exist two lines parallel to each other. Assume that $F_{0}$ has $k$ singular points with multiplicities $m_{1},\cdots,m_{k}$, where $2\leq m_{i} \leq d$. Then we have two equalities as follows: \begin{equation} \label{4.4}
\sum_{i=1}^{k} \binom{m_{i}}{2}=\binom{d}{2},
\end{equation}
\begin{equation} \label{4.5}
0=\chi (\U)=1-d+\sum _{i=1}^{k}(m_{i}-1).
\end{equation}

If there exists a $d$-fold singular point in $V$, then $f=(x_{1}-a)^{d}+(x_{2}-b)^{d}$, where $(a,b)$ is the coordinate of this singular point in $\mathbb{C}^{2}$. Therefore, up to a change of coordinates, $f$ is a homogeneous polynomial.

If there is no $d$-fold point, then $m_{i}< d$.  Using equation (\ref{4.4}), we get \begin{center}
$d(d-1)=\sum_{i=1}^{k} m_{i}(m_{i}-1)<d\sum _{i=1}^{k}(m_{i}-1)$,
\end{center}
which implies $d-1<\sum _{i=1}^{k}(m_{i}-1)$. This contradicts equation (\ref{4.5}).
\end{proof}

\section{Divisibility Results}
In this section, we obtain a divisibility result for Alexander polynomials via the Sabbah specialization complex. The proof of our result is similar to that of Theorem 4.2 in \cite{LM2}.

\subsection{Proof of Theorem 1.3}
Fix a Whitney b-regular stratification of $F_{0}$. Let $\ST$ denote a stratum in this stratification with complex dimension $s$.

We use the Lefschetz hyperplane section theorem and induction down on $i$. The beginning of the induction is the characterization of the top Alexander polynomial: the prime factors of $\delta_{n}(t)$ are among the prime factors of local polynomials $\Delta^{\ST}_{q}(t)$ corresponding to strata $\ST\subset F_{0}$, with $0\leq s=\dim\ST \leq n$ and $n-2s\leq q \leq n-s$.

By the compactly supported hypercohomology long exact sequence corresponding to the inclusion of strata of $F_{0}$, the polynomial associated to $H_{c}^{n+1}(F_{0},\psi^{S}_{f}\Q)$ divides the product of polynomials associated to $H_{c}^{n+1}(\ST,(\psi^{S}_{f}\Q)\vert_{\ST})$, where $\ST$ runs over the strata of $F_{0}$.

Next, we need the following lemma proved by Maxim.
\begin{lem}(\cite{LM2}, Lemma 4.3) \label{l5.1}
The prime factors of the polynomial associated to $H_{c}^{i+1}(\ST,\mathcal{R}^{\bullet op}\vert_{\ST})$ must divide one of the local Alexander polynomials  $\Delta^{\ST}_{q}(t)$ associated to the link pair of an arbitrary point in $\ST$, in the range $0\leq q \leq n-s$ and $0\leq i-q \leq 2s$.
\end{lem}

Lemma \ref{l2.9} shows that $\R\vert_{F_{0}}\cong \psi^{S}_{f}{\Q}$. Note that all the prime factors of $\Delta^{\ST} _{q}(t)$ are roots of unity by general (Hodge) theory, so $\overline{\Delta^{\ST} _{q}(t)}=\Delta^{\ST} _{q}(t)$, where $\overline{\ast}$ represents the involution $t \rightarrow t^{-1}$. Then, Lemma \ref{l5.1} (for $i=n$), composed with the involution, gives that the prime factors of $\delta_{n}(t)$ are among the prime factors of local polynomials $\Delta^{\ST}_{q}(t)$ corresponding to strata $\ST$, with $0\leq s=\dim\ST \leq n$ and $n-2s\leq q \leq n-s$.

We denote the Alexander polynomial of $V$ by $\delta^{V}_{i}(t)$ and call $\delta^{V}_{n}(t)$ the top Alexander polynomial of $V$.
 Let $1\leq k <n$ be fixed. Consider $L=\mathbb{CP}^{n-k+1}$  a generic codimension  $k$ linear subspace of $\CP$, so that $L$ is transversal to $V\cup H$. Then $W=L\cap V$ is a $(n-k)$-dimensional, degree $d$, reduced hypersurface in $L$, which is transversal to the hyperplane at infinity $H\cap L$ of $L$. Moreover, by the transversality assumption, $L\cap F_{0}$ has a Whitney stratification induced from $F_{0}$, with strata of the form $\ST\cap L$, for $\ST$ a stratum in $F_{0}$.

By applying the Lefchetz hyperplane section theorem (\cite{D1}, (1.6.5)) to $\U=\CP-(V\cup H)$ by intersection with $L$, we get that the inclusion $\U\cap L \hookrightarrow\U$ is a homotopy $(n-k+1)$-equivalence. Therefore the homotopy type of $\U$ is obtained from $\U\cap L$ by adding cells of dimension $>n-k+1$. Hence the same is true for the corresponding infinity cyclic covers. Therefore, $H_{i}((\U\cap L)^{c})\cong H_{i}(\U^{c})$  for  $i\leq n-k$, hence $\delta^{W}_{n-k}(t)=\delta^{V}_{n-k}(t)$.

Note that $\delta^{W}_{n-k}(t)$ is the top Alexander polynomial of $W$ as a hypersurface in $L=\mathbb{CP}^{n-k+1}$, therefore, by induction, the prime factors of $\delta_{n-k}(t)$ are among the prime factors of local polynomials $\Delta^{\ST\cap L}_{q}(t)$ corresponding to stratum $\ST\cap L$, $0\leq r=\dim(\ST\cap L) \leq n-k$ and $n-k-2r\leq q \leq n-k-r$. Using the fact that the link pair of a stratum $\ST\cap L$ in $L\cap F_{0}$ is the same as the link pair of $\ST$ in $F_{0}$, the conclusion follows by reindexing (replace $r$ by $s-k$, where $s=\dim\ST$).

\begin{remark} For $x\in \ST$, let $F_{x}$ be the local Milnor fibre of $f$ at $x$. Then $H_{q}(F_{x})$ is a torsion $\Gamma$-module, induced by the inclusion of the link pair of $x$ into $(\CN,F_{0})$, and $\Delta^{\ST} _{q}(t)$ is the corresponding local Alexander polynomial for $H_{q}(F_{x})$. So, by stalk calculation and the superduality isomorphism (see the proof of Lemma 4.3 in \cite{LM2}, or Proposition 3.12 in \cite{Bu2}), 
\begin{equation} \label{5.1}
\mathcal{H}^{q+1}(\psi^{S}_{f}\Q)_{x}\cong  \mathcal{H}^{q+1}(Rl_{\ast}\K)_{x} \cong \overline{ H_{q}(F_{x})},
\end{equation} where $\overline{\ast}$ represents the involution $t \rightarrow t^{-1}$.

 Let $\gamma$ be the monodromy action on the $\Q$-vector space $H_{q}(F_{x})$ induced by the local Milnor fibration. If $\gamma_{c}$ is the monodromy action on $H^{q}(F_{x})$ associated to the nearby cycle functor, then (\cite{DS2})  \begin{center}
$\gamma_{c}= \text{}^{T}\gamma^{-1}$.
\end{center}
 where  $^{T}\ast$ represents the transpose, and $\ast^{-1}$ is the involution. By equation (\ref{5.1}), we have the following $\Gamma$-module isomorphism:
\begin{center}
$\mathcal{H}^{q}(\psi_{f}\Q)_{x}  \cong \text{}^{T}  \mathcal{H}^{q+1}(\psi^{S}_{f}\Q)_{x}$.  
  \end{center}   
In particular, since all the prime factors of local Alexander polynomials are roots of unity, the operations of taking the involution and transpose do not change the Alexander polynomial.
So $H^{2n+1-i}_{c}(\psi^{S}_{f}\Q)$ and $H^{2n-i}_{c}(\psi_{f}\Q)$ give the same Alexander polynomial, and we can realize the Alexander polynomials by nearby cycles. 
\end{remark}
 
\subsection{Application}
Recall that for $f$ transversal at infinity,  $F_{0}\sim \vee_{\mu} S^{n}$. Let $i$ be the inclusion of $F_{0}$ into $\CN$. Then \begin{center}
$\mathcal{D}( i^{!}\Q_{\CN}) \cong i^{ \ast}\Q_{\CN}[2n+2]= \Q_{F_{0}}[2n+2]$.
\end{center} 
Therefore, \begin{center}
$H^{2n+2-i}_{c}(F_{0},i^{ !}\Q_{\CN}) \cong H^{i}(F_{0})=\left\{ \begin{array}{ll}
\Q, & i=0, \\
\Q^{\mu}, & i=n, \\
0, & \text{otherwise.}\\
\end{array}\right.$ 
\end{center} 
Consider the distinguished triangle $ i^{ !}\Q_{\CN} \rightarrow  \phi_{f}\Q[-1]  \overset{var}{\to}  \psi_{f}\Q[-1] \overset{[1]}{\to}$ (\cite{D2}, Remark 4.2.12). By applying the compactly supported hypercohomoly functor to this triangle, we have $\Q$-vector space isomorphisms: \begin{center}
$H^{2n-i}_{c}(F_{0},\psi_{f}\Q)\cong H^{2n-i}_{c}(F_{0},\phi_{f}\Q)$ for $1\leq i< n-1$
\end{center} and a long exact sequence:\begin{equation} \label{5.2}
0\rightarrow H^{n}_{c}(F_{0},\phi_{f}\Q)\rightarrow H^{n}_{c}(F_{0},\psi_{f}\Q)\rightarrow \Q^{\mu}\rightarrow H^{n+1}_{c}(F_{0},\phi_{f}\Q)\rightarrow H^{n+1}_{c}(F_{0},\psi_{f}\Q)\rightarrow 0 
\end{equation}

\subsubsection{Isolated singularities}
Assume that $F_{0}$ has only isolated singular points, Remark \ref{r4.1} shows that $ H_{n}(F)\cong H^{n}_{c}(F_{0},\psi_{f}\Q)$. So we get the same divisibility results as in \cite{DN}, 2.16:
\begin{center}
$\delta_{n}(t)\mid (t-1)^{\mu}\prod_{p\in Sing(F_{0})}\Delta_{p}(t)$,
\end{center} where $ \Delta_{p}(t)$ is the top local Alexander polynomial for $p\in Sing(F_{0})$.

\subsubsection{Manifold singularity locus}
Assume that $Sing(F_{0})$ itself is a stratum with dimension $s$ in the Whitney stratification of $F_{0}$, denoted by $\ST$. Then  the vanishing cycle $\phi_{f}\Q$ is a local system  placed in degree $n-s$ over $\ST$. Let $\Delta(t)$ be the top local Alexander polynomial for arbitrary point in $\ST$.  

If this local system is a constant sheaf (e.g., $\ST$ is simply connected), then \begin{center}
 $\delta_{i}(t)=(\Delta(t))^{\nu_{i-n+s}}$, for  $n-s\leq i<n-1$;
\end{center}
 \begin{center}
  $\delta_{n-1}(t)\mid (\Delta(t))^{\nu_{s-1}}$, for $i=n-1$;
\end{center}
 \begin{center}
  $\delta_{n}(t)\mid (t-1)^{\mu}(\Delta(t))^{\nu_{s}}$, for $i=n$,
\end{center} where $\nu_{i-n+s}$ is the $(i-n+s)$-th betti number of $\ST$.

In particular, if $\ST$ is $(s-1)$-connected (e.g., $\ST$ is a smooth complete intersection, see \cite{D1}, (5.3.24)), then $\ST$ is homotopy equivalent to a bouquet of $s$-spheres, $\ST\sim \vee_{\eta} S^{s}$.  Assume $s> 1$. Then $\ST$ is simply connected, so, in this case, \begin{center}
  $\delta_{i}(t)=\Delta(t)$, for $i=n-s$;

  $\delta_{i}(t)=1$, for $n-s<i<n$;

  $\delta_{n}(t)\mid (t-1)^{\mu}(\Delta(t))^{\eta}$, for $i=n$.
\end{center}
\begin{example} Assume $V$ has only two irreducible components $V_{1}$ and $V_{2}$, which intersects transversally with each other. If $V_{1}$ and $V_{2}$ are both smooth and $n\geq 3$,
 then $\delta_{i}(t)=1$ for $2\leq i \leq n-1$.
\end{example}

\subsubsection{First non-vanishing Alexander module}
Here we use the vanishing cycles to recover Libgober's result (\cite{L3}, Lemma 1.5), which states that: $H_{i}(\U^{c})=0$ for $0<i<n-k$, where $k\geq 0$ is the complex dimension of the singular part of $F_{0}$. In fact, Libogber showed that $\pi_{i}(\U^{c})=0$ for $0<i<n-k $ by using the Lefschetz hyperplane section theorem.

The vanishing cycle complex $\phi_{f}\Q$ has support on $Sing(F_{0})$. Since $Sing(F_{0})$ is a $k$-dimensional affine variety and $\phi_{f}\Q$ is a perverse sheaf on $Sing(F_{0})$, $H^{2n-i}_{c}(F_{0},\phi_{f}\Q)=0$ if $i\notin [n-k,n]$ by Artin's vanishing results. Therefore $H_{i}(\U^{c})\cong H^{2n-i}_{c}(F_{0},\psi_{f}\Q)=0$ for $1\leq i< n-k$.
    
We also get that $H_{i}(F)\cong H_{i}(\U^{c})=0$ if $1\leq i< n-k$, for $k \geq 0$. In fact, the generic fibre is actually $(n-k-1)$-connected. Dimca and Libgober showed (\cite{DL}) that there is a inclusion from the generic fibre $F$ to $\U^{c}$, which induces a homotopy $n$-equivalence.  Then $\pi_{i}(F)\cong \pi_{i}(\U^{c})$ for $0\leq i<n$. In particular, $\pi_{i}(F)\cong \pi_{i}(\U^{c})=0$ for $0<i<n-k$.

\subsubsection{The eigenvalue decomposition of the nearby cycles and vanishing cycles}
 The roots of $\delta_{i}(t)$ are roots of unity of order $d$ for $0\leq i \leq n$. Then we have the following $\Q$-vector space isomorphisms for $1\leq i <n$: \begin{center}
$ H_{i}(\U^{c},\C)\cong H^{2n-i}_{c}(F_{0},\psi_{f}\C)\cong \oplus_{\lambda^{d}=1}H^{2n-i}_{c}(F_{0},\psi_{f,\lambda}\C) $,
\end{center}  where $\psi_{f,\lambda}\C$  is the subsheaf complex of $\psi_{f}\C$ with eigenvalue $\lambda$. This shows that if $\lambda^{d}\neq 1$, $H^{2n-i}_{c}(F_{0},\psi_{f,\lambda}\C)\cong H^{2n-i}_{c}(F_{0},\phi_{f,\lambda}\C)=0 $ for $1\leq i <n$.

\subsubsection{Rational homology manifold}
   Assume that $F_{0}$ is a rational homology manifold. This is equivalent to say that all the local links are rational homology spheres, i.e., all the local monodromy operators do not have $1$ as eigenvalue in positive degrees.

\begin{cor} Assume that the polynomial $f: \CN\rightarrow \C$ is transversal at infinity. If $F_{0}$ is a rational homology manifold, then $\delta_{i}(1)\neq 0$ for all $0<i<n$.
\end{cor}
\begin{proof} Consider the distinguished triangle $\Q_{F_{0}}\rightarrow \psi_{f,1}\Q \rightarrow \phi_{f,1}\Q \overset {[1]}{\to}$. All the local monodromy operators do not have $1$ as eigenvalue in positive degrees, if and only if, $\phi_{f,1}\Q=0$. So $\Q_{F_{0}}\cong  \psi_{f,1}\Q$. 

Since all the local links are spheres, Poincar\'{e} duality holds for $F_{0}$.  Recall that $F_{0}\sim \vee_{\mu} S^{n}$, then \begin{center}
$H^{2n-i}_{c}(F_{0},\psi_{f,1}\Q) \cong H^{2n-i}_{c}(F_{0}) \cong H^{i}(F_{0})=\left\{ \begin{array}{ll}
\Q, & i=0, \\
\Q^{\mu}, & i=n,\\
0, & \text{otherwise.}\\
\end{array}\right.$ 
\end{center} Therefore $\delta_{i}(1)\neq 0$ for all $0<i<n$.  
\end{proof}

  We can improve this result as follows (compare with the similar proof of Maxim's result \cite{LM2}, Proposition 2.1):
 \begin{prop} \label{p5.5} Assume that the polynomial $f: \CN\rightarrow \C$ is transversal at infinity. If $F_{0}$ is a rational homology manifold, then $V$ is irreducible and $\delta_{i}(1)\neq 0$ for all $0<i\leq n$.
\end{prop}
 \begin{proof} Since all the local links are rational homology spheres, Poincar\'{e} duality holds for $F_{0}$.  Recall that $F_{0}\sim \vee_{\mu} S^{n}$, so $H^{i}_{c}(F_{0},\Q)=\left\{ \begin{array}{ll}
\Q, & i=2n, \\
\Q^{\mu}, & i=n, \\
0, & \text{otherwise.}\\
\end{array}\right.$

 By the compactly supported hypercohomology long exact sequence for the inclusions of $\U$ and $F_{0}$ into $\CN$, we get: \begin{center}
 $\cdots \rightarrow H^{i}_{c}(\U,\Q)\rightarrow H^{i}_{c}(\CN,\Q) \rightarrow H^{i}_{c}(F_{0},\Q)\rightarrow \cdots$
 \end{center} Thus $H^{i}_{c}(\U,\Q)=\left\{ \begin{array}{ll}
\Q, & i=2n+1,2n+2, \\
\Q^{\mu}, & i=n+1, \\
0, & \text{otherwise.}\\
\end{array}\right.$

Again, by Poincar\'{e} duality, we get \begin{equation} \label{5.3}
H^{i}(\U,\Q)=\left\{ \begin{array}{ll}
\Q, & i=0,1, \\
\Q^{\mu}, & i=n+1, \\
0, & \text{otherwise.}\\
\end{array}\right.
\end{equation}  Recall that  $H_{1}(\U,\Z)=\Z^{r}$, where $r$ is the number of the irreducible components of $V$.  So $r=1$ and $V$ is irreducible.

 Recall Milnor's long exact sequence (\cite{DN}): \begin{center}
 $ \cdots \rightarrow H_{i}(\U^{c},\Q)\rightarrow H_{i}(\U^{c},\Q) \rightarrow H_{i}(\U,\Q)\rightarrow H_{i-1}(\U^{c},\Q) \rightarrow \cdots$
 \end{center} where the first morphism is multiplication by $t-1$. Then the equality (\ref{5.3}) gives that $\delta_{i}(1)\neq 0$ for all $0<i\leq n$.
\end{proof}
 
\begin{remark} \label{r5.6} 
This proposition shows that if $F_{0}$ is a rational homology manifold, we can realize the Alexander polynomial $\delta_{i}(t)$ by vanishing cycles instead of nearby cycles for $0<i\leq n$. In particular, since in this case $\phi_{f,1}\Q=0$, the long exact sequence (\ref{5.2}) breaks into short exact sequence, so we have the following $\Q$-vector space isomorphisms: \begin{center}
$H^{2n-i}_{c}(F_{0},\psi_{f}\Q)\cong H^{2n-i}_{c}(F_{0},\phi_{f}\Q)$ for $1\leq i< n$;
\end{center} \begin{center}
$ H^{n}_{c}(F_{0},\psi_{f}\Q)\cong H^{n}_{c}(F_{0},\phi_{f}\Q)\oplus \Q^{\mu}$.
\end{center}
\end{remark}
\subsection{Maxim's conjecture}
We already have a general divisibility result which restricts the prime factors of the global Alexander polynomials to those of the local Alexander polynomials of link pairs around the singular strata. More refined results can be expected under some good assumptions. Maxim posed a conjecture related to this topic (\cite{LM1}, Conjecture 4.1): When $\psi_{f}\Q$ splits, i.e., $\psi_{f}\Q$ is a finite sum of Deligne-Goresky-MacPherson sheaves (see \cite{D2}, Definition 5.4.7), the prime factors of the global Alexander polynomials are only among those of the top local Alexander polynomials of link pairs around the singular strata.

Fix a Whitney b-regular stratification for $F_{0}$. Let $\ST$ be a stratum in this stratification with complex dimension $s$. For $x\in \ST$, we have the local Milnor fibre $F_{x}$, the boundary of Milnor fiber $\partial F_{x}$ and the link $K_{x}$.
Then $\mathcal{H}^{n-s}((\psi_{f}\Q)_{\vert \ST})$ is a local system on $\ST$ with stalk  the $\Q$-vector space $H^{n-s}(F_{x})$. Let $\K_{n-s}(\ST)$ denote this local system. Let $\overline{\ST}$ be the closure of $\ST$ in $F_{0}$, hence a $2s$-(real) dimensional pseudomanifold, where $\ST$ is its dense smooth part. Then we can define the $\overline{m}$-perversity intersection chain sheaf $IC_{\overline{m}}(\overline{\ST},\K_{n-s}(\ST))$.

We say that $\psi_{f}\Q$ splits if there is a decomposition in the category of (shifted) perverse sheaves:
\begin{center}
$\psi_{f}\Q\cong \oplus_{\ST}IC_{\overline{m}}(\overline{\ST},\K_{n-s}(\ST))[-s-n]$,
\end{center} where $\ST$ runs over all the strata of $F_{0}$.

\begin{remark} In fact, Maxim defined the split (\cite{LM1}) for the Sabbah specialization complex rather than for nearby cycles. By using the forgetful functor, it follows that the Sabbah specialization complex splits if and only if its underlying nearby cycles split. 
\end{remark}

\begin{lem}(\cite{LM1}, Theorem 3.3) $\psi_{f}\Q$ splits if, and only if, for all $x\in \ST \subset Sing(F_{0})$ (where $\ST$ is a stratum in $Sing(F_{0})$ with complex dimension $s$), the natural homomorphism \begin{center}
$H_{n-s}(F_{x})\rightarrow H_{n-s}(F_{x},\partial F_{x})$
\end{center}  induced by the inclusion map is an isomorphism.
\end{lem}

\begin{lem} \label{l5.9} Let $f:(X,x)\rightarrow (\C,0)$ be a polynomial map, where $X$ is a small open ball around the point $x$ in a $(n+1)$-complex dimensional smooth algebraic variety. Set $X_{0}=f^{-1}(0)\cap X$.
Let $F$ be the corresponding Milnor fibre, and let $T^{n}$ be the corresponding monodromy automorphism on $H^{n}(F)$ induced by the Milnor fibration. If the map $H^{n}(F,\partial F)\rightarrow H^{n}(F)$ is an isomorphism, then $T^{n}$ does not have $1$ as an eigenvalue.
\end{lem}

\begin{proof} Let $i$ be the inclusion of point $x$  to $X_{0}$ and $\Q$ be the constant sheaf on $X$. We have the following commutative diagram ($n\geq 1$):
\begin{center}
$\xymatrix{
H^{n}(X_{0},Ri_{\ast}i^{!}\psi_{f}\Q)  \ar[r]^{(1)} \ar[d]^{i_{\ast}i^{!} can}  & H^{n}(X_{0},\psi_{f}\Q)  \ar[d]^{can} \\
H^{n}(X_{0},Ri_{\ast}i^{!}\phi_{f}\Q)  \ar[r]^{(2)}                          & H^{n}(X_{0},\phi_{f}\Q)
}$
\end{center}
Here (1) is an isomorphism since this is exactly the map $H^{n}(F,\partial F)\rightarrow H^{n}(F)$ (\cite{DS}); the map $can$ is naturally an isomorphism since $X_{0}$ is contractible. Note that these four cohomology groups are all isomorphic (as $\Q$-vector spaces) to $H^{n}(F)$, so (2) and $i_{\ast}i^{!} can$ are also  isomorphisms.

Recall the commutative diagram (\cite{D2}, Remark 4.2.12)
\begin{center}
$\xymatrix{
H^{n}(X_{0},Ri_{\ast}i^{-1}\phi_{f}\Q)  \ar[r]^{i_{\ast}i^{-1} var} \ar[d]^{T^{n}-id}  & H^{n}(X_{0},Ri_{\ast}i^{-1}\psi_{f}\Q)  \ar[d]^{i_{\ast}i^{-1} can} \\
H^{n}(X_{0},Ri_{\ast}i^{-1}\phi_{f}\Q)  \ar[r]^{id}                                    & H^{n}(X_{0},Ri_{\ast}i^{-1}\phi_{f}\Q)
}$
\end{center}
Since $\phi_{f}\Q[n]$ is a self dual, perverse sheaf, and $\mathcal{D} i_{\ast}i^{!} can =i_{\ast}i^{-1} var \mathcal{D}$, the map $i_{\ast}i^{-1} var$ in the diagram is just the dual of $i_{\ast}i^{!} can$, hence an isomorphism. Now $i_{\ast}i^{-1} var$, $i_{\ast}i^{-1} can$ and $id$ in the diagram are all isomorphisms, then $T^{n}-id$ is also an isomorphism. Therefore $T^{n}$ does not have 1 as eigenvalue. 
\end{proof}

\begin{prop} \label{p5.10} Assume that the polynomial $f: \CN\rightarrow \C$ is transversal at infinity.
If $\psi_{f}\Q$ splits, then $F_{0}$ is a rational homology manifold. Moreover, $V$ is irreducible and $\delta_{i}(1)\neq 0$ for all $0<i\leq n$.
\end{prop} 
\begin{proof} In order to prove this proposition, we need to show that for all $x\in Sing(F_{0})$,  the link $K_{x}$ is a rational homology sphere. Assume $x\in \ST$, where $\ST$ is a stratum in $Sing(F_{0})$. Let us prove this claim by induction down on the dimension of $\ST$. Let $k$ be the complex dimension of $Sing(F_{0})$. 

For $s=k$, $x$ is in the top dimension stratum of $Sing(F_{0})$, the link pair of $x$ can be viewed as the link pair for an isolated singular point. Then the Milnor fibre $F_{x}$ is homotopy equivalent to a bouquet of $(n-k)$-spheres. In this case, $H^{n-k}(F_{x},\partial F_{x})\rightarrow H^{n-k}(F_{x})$ being an isomorphism implies that the link $K_{x}$ is a homology sphere.

Assume the claim is true when $\dim \ST>0$ for induction. 

 Let the point $x$ be a 0-dimensional stratum in $Sing(F_{0})$. If $y\in K_{x}$, then $y$ belongs to a stratum in $F_{0}$ with dimension $\geq 1$, and the local link of $y$ in $K_{x}$ is a rational homology sphere by induction. Therefore $K_{x}$ is a rational homology manifold, and Poincar\'{e} duality holds for $K_{x}$. Recall that $K_{x}$ is $(n-2)$-connected (\cite{M}), so by Poincar\'{e} duality, we have \begin{center}
$H^{i}(K_{x})=0$,  for $n<i<2n-1$.
\end{center} Then, by Alexander duality, \begin{center}
$H_{j}(S^{2n+1}\setminus K_{x})=H^{2n+1-j}(S^{2n+1}, K_{x})=H^{2n-j}(K_{x}) =0$, for $1<j<n$.
\end{center}  where $S^{2n+1}$ is a small sphere centered at $x$ in $\CN$. Then the Wang sequence: \begin{center}
$\cdots \to H_{j+1}(S^{2n+1}\setminus K_{x}) \to H_{j}(F_{x}) \overset{T-id}{\to} H_{j}(F_{x}) \to H_{j}(S^{2n+1}\setminus K_{x}) \to \cdots$,
\end{center} shows that $T^{i}$ does not have $1$ as eigenvalue for $0<i<n$. It follows from Lemma \ref{l5.9} that if $\psi_{f}\Q$ splits, $T^{n}$ also does not have $1$ as eigenvalue. Thus $K_{x}$ is a rational homology sphere.   Then apply Proposition \ref{p5.5}.
\end{proof}

\begin{remark} This proposition shows that, if $\psi_{f}\Q$ splits, then $\phi_{f,1}\Q=0$ ($F_{0}$ is a rational homology manifold), hence $\psi_{f}\Q= \phi_{f}\Q \oplus \Q_{F_{0}}$.  Therefore, $\phi_{f}\Q$ also splits : \begin{center}
$\phi_{f}\Q\cong\oplus_{\ST}IC_{\overline{m}}(\overline{\ST},\K_{n-s}(\ST))[-s-n]$,
\end{center}  where $\ST$ runs over all the strata of $Sing(F_{0})$. 
\end{remark}

\textbf{Proof of Theorem \ref{t1.4}.} Assume that $\psi_{f}\Q$ splits, then  $\phi_{f}\Q$ also splits, hence \begin{center}
$H^{2n-i}_{c}(F_{0},\phi_{f}\Q)\cong\oplus_{\ST} H^{2n-i}_{c}(\overline{\ST},IC_{\overline{m}}(\overline{\ST},\K_{n-s}(\ST)[-s-n]))$.
\end{center} where $\ST$ runs over all the strata of $Sing(F_{0})$.

Proposition \ref{p5.10} shows that $F_{0}$ is a rational homology manifold. Then, according to Remark \ref{r5.6}, we can realize the Alexander polynomials by vanishing cycles instead of nearby cycles for $0<i\leq n$. If a prime element $\gamma\in \Gamma$ divides $\delta_{i}(t)$, then there exists a stratum $\ST$ in $F_{0}$ with dimension $s\geq n-i$, such that $\gamma$ divides the Alexander polynomial associated with $H^{n-s-i}_{c}(\overline{\ST},IC_{\overline{m}}(\overline{\ST},\K_{n-s}(\ST)))$. Fix this stratum $\ST$, and write $\K_{n-s}(\ST)$ for short as $\K_{n-s}$.
 
  Denote  $\G=IC_{\overline{m}}(\overline{\ST},\K_{n-s})$. In order to prove the theorem, we need to show that the Alexander polynomial associated with $H^{n-s-i}_{c}(\overline{\ST},\G)$ is some power of $\Delta^{\ST}_{n-s}(t)$, where $\Delta^{\ST}_{n-s}(t)$ is the top local Alexander polynomial associated to the singular stratum $\ST$. This follows from the following more general fact:
 
\textit{Claim}:  For any $x \in \overline{\ST} $ and any $j\in \Z$,  the corresponding Alexander polynomial of $\mathcal{H}^{j}(\G)_{x}$  is some power of $\Delta^{\ST}_{n-s}(t)$ (the power is 0 except for finitely many $j$).

 \textit{Proof of the claim.}  Assume $x\in \V$, where $\V$ is a stratum in $\overline{\ST}$. Let us prove the claim by induction down on the complex dimension of $\V$.
 
 If $\dim\V=s$, then $\V=\ST$ and $x\in \ST$.  $\K_{n-s}$ is a local system on $\ST$, so the claim is true.
      
  Assume the claim is true when $\dim\V>0$ for induction. Now let us prove the claim when the point $x$ itself is a stratum in $\overline{\ST}$. By Proposition 4.1.21 in \cite{B}, we have \begin{center} 
  $\mathcal{H}^{-j}(\G)_{x}\cong \left\{ \begin{array}{ll}
IH^{\overline{m}}_{j-1}(K_{x},\K_{n-s})\cong H^{-j}(K_{x},\G\vert_{K_{x}}),   & i\geq 2s-\overline{m}(2s) \\
0, & i<2s-\overline{m}(2s)\\
\end{array}\right.$
  \end{center} where $K_{x}$  is the link of $x$ in $\overline{\ST}$. 
  For arbitrary $ y \in K_{x}$, $y$ belongs to a stratum in $\overline{\ST}$ with dimension $\geq 1$. Then, by induction, the Alexander polynomial of $\mathcal{H}^{j}(\G)_{y}$ is some power of $\Delta^{\ST}_{n-s}(t)$ (the power is 0 except for finitely many $j$). Using the compactly supported hypercohomology long exact sequence on the strata of $K_{x}$ and Lemma \ref{l5.1}, it is easy to see that the Alexander polynomial of $H^{-j}(K_{x},\G\vert_{K_{x}}) $  is some power of $\Delta^{\ST}_{n-s}(t)$, hence so is $\mathcal{H}^{-j}(\G)_{x}$. (Note that the involution does not change the local Alexander polynomial.)  The claim is proved.
  
 Since the claim is true, using again the compactly supported hypercohomology long exact sequence on the strata of $\overline{\ST}$ and Lemma \ref{l5.1}, we get that the Alexander polynomial associated with  $H^{n-s-i}_{c}(\overline{\ST},\G)$ is some power of $\Delta^{\ST}_{n-s}(t)$. Theorem \ref{t1.4} is thus proved.
  
\begin{remark} Assume that $f$ is transversal at infinity and $\psi_{f}\Q$ splits. Let $\ST$ be a stratum in $Sing(F_{0})$ with complex dimension $s\geq 1$. 
If the local system $\K_{n-s}(\ST)$ is a constant sheaf (e.g., $\ST$ is simply connected), then $IH^{\overline{m}}_{0}(\overline{\ST},\Q)=\Q$, so  we get a lower bound for $\delta_{n-s}(t)$: \begin{center}
$\Delta^{\ST}_{n-s}(t)\mid \delta_{n-s}(t)$.
\end{center} Moreover, the roots of $\Delta^{\ST}_{n-s}(t)$ are roots of unity of order $d$ except $1$, and the corresponding top local Alexander module is semi-simple.
\end{remark}
It is natural to ask what are the multiplicities of the roots of global Alexander polynomials.  The next corollary gives a partial answer to this question. 

 Assume the following condition is satisfied: All the strata $\ST\subset F_{0}$ are simply connected and their closures $\overline{\ST}$ are rational homology manifolds.  Then $\psi_{f}\Q$ splits (\cite{LM1}, Theorem 3.11). Since $\overline{\ST}$ is a rational homology manifold, $IH^{\overline{m}}_{i}(\overline{\ST},\Q)\cong H_{i}(\overline{\ST},\Q)$. Set $\nu_{j}(\overline{\ST})=rankH_{j}(\overline{\ST})$. Then we get the following corollary:
\begin{cor} \label{c5.13} Assume that the polynomial $f: \CN\rightarrow \C$ is transversal at infinity. If all the strata $\ST\subset F_{0}$ are simply connected and their closures $\overline{\ST}$ are rational homology manifolds, then \begin{center}
 $\delta_{i}(t)=\prod_{\ST,\dim \ST=s\geq n-i}{(\Delta^{\ST}_{n-s}(t))}^{\nu_{i+s-n}(\overline{\ST})}$,    for $0<i<n$;
 \end{center}
 \begin{center}
 $ \delta_{n}(t)\mid \prod_{\ST}{(\Delta^{\ST}_{n-s}(t))}^{\nu_{s}(\overline{\ST})}$,    for $i=n$,
 \end{center} where $\ST$ runs over all the strata of $Sing(F_{0})$.
\end{cor}

\begin{example} Consider the homogeneous polynomial in $\CN$ defined by \begin{center}
$f(x_{1},\cdots,x_{n+1})=x_{1}^{d}+\cdots+x_{n+1-k}^{d}$ with $d>1$, $0\leq k<n+1$.
\end{center}  $\psi_{f}\Q$ splits, if and only if, $d=2$ and $n+1-k$ is odd. In fact, the stratification of $F_{0}$ is $F_{0} \supset \ST$, where $\ST=\lbrace (x_{1},\cdots,x_{n+1})\in \CN\mid x_{i}=0 \text{  for } 1\leq i\leq n+1-k \rbrace$. Then the stratum $\ST$ is isomorphic with $\C^{k}$. For any $ x \in \ST$, $K_{x}$ is a homology sphere if and only if $d=2$ and $n+1-k$ is odd (\cite{D1}, (3.4.10)). 

When $d=2$ and $n+1-k$ is odd, Corollary \ref{c5.13} gives that $\delta_{n-k}(t)={\Delta^{\ST}_{n-k}(t)}=t+1$ and $\delta_{i}(t)=1$ for $i\neq n-k$, $0<i\leq n$.
\end{example}

\section{Mixed Hodge Structures}
Dimca and Libgober showed that for polynomials transversal at infinity, there exist MHS on the Alexander modules of the hypersurface complement (\cite{DL}, Theorem 1.5). In this section, we will recover this MHS by using the nearby cycle functor in the category of  Mixed Hodge Modules. This is the most important application of realizing Alexander modules by the nearby cycles and also a very natural way to get this MHS. 

Saito  associated (\cite{S}) to a complex algebraic variety $X$ an abelian category $\MHM(X)$, the category of algebraic mixed Hodge modules on $X$. There is a forgetful functor  $$\mbox{rat}: \MHM(X)\rightarrow Perv_{\Q}(X),$$ assigning a perverse sheaf to a mixed Hodge module. This functor extends to the derived level, and we have:

\begin{thm}(\cite{S}, Theorem 1.3) For each complex algebraic  variety X, we have $D^{b}\MHM(X)$, the bounded derived category of complexes of mixed Hodge modules with a forgetful functor \begin{center}
$\mbox{rat}: D^{b}\MHM(X)\rightarrow D^{b}_{c}(X,\Q)$,
\end{center}
which associates to mixed Hodge modules their underlying $\Q$-complexes, such that \begin{center}
$\mbox{rat}(\MHM(X))\subseteq Perv_{\Q}(X)$, i.e., 
$\mbox{rat}(H^{j}(\M^{\bullet}))= \text{}  ^{p}\mathcal{H}^{j}(rat(\M^{\bullet}))$.
\end{center} Moreover, the functor $\psi_{f}$ is naturally lifted to a functor on $D^{b}\MHM(X)$, in a compatible way with the corresponding functor on the underlying rational complex, i.e., $^{p}\psi_{f}\circ \mbox{rat}=\mbox{rat}\circ \psi_{f}$, where $^{p}\psi_{f}=\psi_{f}[-1]$. 
\end{thm}
Note that the cohomology groups with compact supports of any complex $\mathcal{M}^{\bullet}\in D^{b}\MHM(X)$ carry mixed Hodge structures.  Let $\Q$ be the constant sheaf in $D^{b}\MHM(\CN)$, then $\psi_{f}\Q\in D^{b}\MHM(F_{0})$ and $H^{2n-i}_{c}(F_{0},\psi_{f}\Q)$ naturally carries a MHS for all $i$. Recall that \begin{center}
 $H_{i}(\U^{c})\cong H^{2n-i}_{c}(F_{0},\psi_{f}\Q)$ for $i<n$.
 \end{center} Thus we obtain a MHS for $H_{i}(\U^{c})$ for $i<n$.  But $H_{n}({\U}^{c})$ is just a quotient $\Q$-vector space of $H^{n}_{c}(F_{0},\psi_{f}{\Q})$, so we can't get a MHS for $H_{n}(\U^{c})$ in this way.   We solve this problem by constructing a transversal pair (see below for a definition).

\subsection{Proof of Theorem 1.5}

Let $V^{\prime}$ be a reduced hypersurface in $\mathbb{CP} ^{n+2}$ and $H^{\prime}$ be the hyperplane at infinity. Assume $V^{\prime} $ intersects $ H^{\prime} $ transversally.  Let $L$ denote a (n+1)-dimensional linear subspace of $\mathbb{CP}^{n+2}$. If $V^{\prime}\cap L =V$, $H^{\prime}\cap L =H$ and $L$ is transversal to $V^{\prime}\cup H^{\prime}$, we call $(V^{\prime},H^{\prime},L)$ a transversal pair for $(V,H)$.  The following lemma shows that transversal pairs always exist for a given $(V,H)$ in $\CP$, with $V$ transversal to $H$.

\begin{lem} \label{l6.2} Let $V= \lbrace \widetilde{f}=0\rbrace$ be a reduced hypersurface in $\CP$ and $H=\lbrace x_{0}=0\rbrace $ be the hyperplane at infinity. Assume $V$ intersects $H$ transversally. Then there exists a transversal  pair $(V^{\prime},H^{\prime},L)$ such that $V^{\prime}\cap L =V$, $H^{\prime}\cap L =H$, i.e., this pair satisfies the following transversality conditions:

(a) $V^{\prime}$ intersects $H^{\prime}$ transversally.

(b) $L$ intersects $V^{\prime}\cup H^{\prime}$ transversally.
\end{lem}
\begin{proof} We will use the notation $x=(x_{0},\cdots,x_{n+1})$. Choose \begin{center}
$V^{\prime}=\lbrace \widetilde{f}=0 \rbrace$,   $H^{\prime}=\lbrace x_{0}-x_{n+2}=0 \rbrace$,  $L=\lbrace x_{n+2}=0 \rbrace$,
\end{center}  where $x_{n+2}$ is a new variable. Here $V^{\prime}$ is defined by the same polynomial as $V$. It is obvious that $V^{\prime}\cap L =V$ and $H^{\prime}\cap L =H$.

 $V^{\prime}$ is the projective cone over $V$ with the vertex point $p=[0,0,\ldots, 0,1]$ and $L$ is the natural embedding of the $(n+1)$-dimensional projective space into $\mathbb{CP} ^{n+2}$.
 
 Fix a Whitney b-regular stratification $\varsigma$ of V. Define\begin{center}
 $\ST\times \C=\lbrace [x,x_{n+2}]\in \mathbb{CP} ^{n+2} \mid x\in \ST, x_{n+2}\in \C \rbrace\subseteq V^{\prime}$ for $\ST\in \varsigma$.
\end{center}   Then we get a stratification for $V'$ by using the strata $\ST \times \C$ and the point $p$.

 With this stratification, it is clear that $H^{\prime}$ intersects $V^{\prime}$ transversally. Indeed $p$ is not in $H^{\prime}$, and any other stratum $\ST \times \C$ is a union of lines transversal to $H^{\prime}$. Then we can stratify $V^{\prime} \cup H^{\prime}$ using the following strata: point $p$, $H^{\prime}-V^{\prime}\cap H^{\prime}$, $\ST\times \C-(\ST\times \C)\cap H^{\prime}$ and $(\ST\times \C) \cap H^{\prime}$. 
 
Next, we prove that $L$ intersects $V^{\prime} \cup H^{\prime}$ transversally (using the stratification described above). For this we notice:

- $p$ is not in $L$;

- $L$ is transversal to any stratum of the form $\ST \times \C$, since $\ST \times \C$ is a union of lines transversal to $L$. Hence $L$ is transversal to any open subset of  $\ST \times \C$;

-$L$ is transversal to $H'$, hence $L$ is transversal to any open subset of  $H'$.

It remains to prove the transversality of $L$ to a stratum of the form $(\ST \times \C) \cap H^{\prime}$ and $\dim \ST >0$.
Consider a point $q$ of the intersection of $L$ with such a stratum. We can assume $q=[0,1,0,\cdots,0]$. This point $q$ is then in the intersection of $\ST$ and $H$ regarded in $\CP$. Since $V$ intersects $H$ transversally, there is a path $a(t)= [a_0(t),1,a_2(t),\cdots,a_{n+1}(t)]$ in $\ST$ such that $a(0)=q$ and  the tangent vector $a'(0)$ to $\ST$ is not in $H$, that is  $a'_0(0) \ne 0$. This path gives rise to a path $b(t)= [a_0(t),1,a_2(t),\cdots,a_{n+1}(t),a_0(t)]$ in  $(\ST \times \C) \cap H ^{\prime}$ such that $b'(0)$ is a tangent vector to this stratum, not contained in the hyperplane
$L$ (since the last component of this vector is not zero).
 \end{proof}

We can similarly define $F_{0}^{\prime}$ , $\U^{\prime}$, $(\U^{\prime})^{c}$ associated to $V^{\prime}$ and $H^{\prime}$.  Let $f^{\prime}$ be the corresponding polynomial in $\C^{n+2}=\mathbb{CP}^{n+2}\setminus H^{\prime}$. Then by applying the Lefschetz hyperplane section theorem to $\U^{\prime}$ and its section $L$, we obtain the isomorphism $\pi_{i}(\U)=\pi_{i}(\U^{\prime})$ for $i\leq n$. Therefore \begin{center}
$H_{i}(\U^{c})\cong H_{i}((\U^{\prime})^{c})\cong H^{2n+2-i}_{c}(F_{0}^{\prime},\psi_{f^{\prime}}\Q)$ for $i\leq n$.
\end{center} Then $H_{i}(\U^{c})$ inherits a MHS from $H^{2n+2-i}_{c}(F_{0}^{\prime},\psi_{f^{\prime}}\Q)$ for $i\leq n$.

In order to finish the proof of Theorem \ref{t1.5}, we need to show that this MHS does not depend on the choice of the transversal pair.

\begin{lem} \label{l6.3} Assume that the polynomial $f: \CN\rightarrow \C$ is transversal at infinity. Then the MHS on $H_{i}(\U^{c})$ ($i\leq n$) inherited from $H^{2n+2-i}_{c}(F_{0}^{\prime},\psi_{f^{\prime}}\Q)$ does not depend on the choice of the transversal pair.
\end{lem}
\begin{proof} Choose $(V_{1}^{\prime},H_{1}^{\prime},L_{1})$ as defined in Lemma \ref{l6.2}. Let $(V_{2}^{\prime},H_{2}^{\prime},L_{2})$ be an arbitrary transversal pair for $(V,H)$.  Without loss of generality, we can assume that we add one new variable $x_{n+3}$ such that $L_{2}=\lbrace x_{n+3}=0\rbrace$, $H_{2}^{\prime}=\lbrace x_{0}=0 \rbrace$ and $V_{2}^{\prime}=\lbrace g=0 \rbrace$, where $g(x_{0},\cdots,x_{n+1},x_{n+3})\vert_{x_{n+3}=0}=\widetilde{f}$.  Construct $V^{\prime\prime}$ and $H^{\prime\prime}$ for $(V_{2}^{\prime},H_{2}^{\prime})$ as we did in Lemma \ref{l6.2}: $V^{\prime\prime}=\lbrace g=0 \rbrace$ and $H^{\prime\prime}=\lbrace x_{0}-x_{n+2}=0 \rbrace$.

Consider the following commutative diagram of spaces and maps:\begin{center}
$\xymatrix{
(V,H) \ar[r]^{i_{2}} \ar[d]^{i_{1}}  & (V_{2}^{\prime},H_{2}^{\prime}) \ar[d]^{i_{2}^{\prime}}  \\
(V^{\prime}_{1},H_{1}^{\prime}) \ar[r]^{i^{\prime}_{1}}         & (V^{\prime \prime},H^{\prime \prime})              
}$
\end{center}
where $i_{1}^{\prime}$ is induced by intersection with the hyperplane $\lbrace x_{n+3}=0 \rbrace$,   and  $i_{2}^{\prime}$ is induced by the intersection with hyperplane $\lbrace x_{n+2}=0 \rbrace$

In this diagram, $i_{2}$ is an inclusion for a transversal pair by the assumption, $i_{1}$ and $i_{2}^{\prime}$ are inclusions for the transversal pairs by Lemma \ref{l6.2}, hence they  all induce isomorphisms on $\pi_{i}(\U)$ for $i\leq n$.  Therefore $i_{1}^{\prime}$ also induces an isomorphism on $\pi_{i}(\U)$ for $i\leq n$, which shows that the resulting MHS on $H_{i}(\U^{c})$ is independent of the choice for $i\leq n$.
\end{proof}

\begin{remark}\label{r6.4} The MHS on $H_{i}(\U^{c})$ inherited from $H^{2n+2-i}_{c}(F_{0}^{\prime},\psi_{f^{\prime}}\Q)$ is same as the one inherited from $H^{2n-i}_{c}(F_{0},\psi_{f}\Q)$ for $i<n$.  
In fact, consider the following commutative diagram of spaces and maps:  \begin{center}
$\xymatrix{
F_{0} \ar[r]^{i^{\prime}}   & F_{0}^{\prime}  \ar[d]           &  F^{\prime}_{0}\setminus F_{0} \ar[l]_{j^{\prime}} \ar[d]^{v} \\
          &  V^{\prime}    & V^{\prime}\setminus L \ar[l]
}$
\end{center}  where $F_{0}^{\prime}=V^{\prime}\setminus H^{\prime}$. Since $L$ is transversal to $V^{\prime}\cup H^{\prime}$, Lemma 6.0.5 in \cite{Sc}, applied to $\F=j^{\prime \ast}\psi_{f^{\prime}}\Q$ on $F^{\prime}_{0}\setminus F_{0}$, gives us that  \begin{center}
$H^{j}_{c}(F_{0}^{\prime},Rj^{\prime}_{\ast}\F)\cong H^{j}(V^{\prime}\setminus L,v_{!}\F)$ for all $j$.
\end{center} Note that $\psi_{f^{\prime}}\Q[n+1]$ is a perverse sheaf, hence so is $\F[n+1]=j^{\prime \ast}\psi_{f^{\prime}}\Q[n+1]$. Moreover, Lemma 6.0.5 in \cite{Sc} shows that \begin{center}
 $v_{!}\F[n+1] \in \text{}^{m} D^{\leq 0}(V^{\prime}\setminus L)$, and $Rj^{\prime}_{\ast}\F[n+1] \in \text{}^{m} D^{\geq 0}(F^{\prime}_{0})$,
\end{center}  where $\text{}^{m} D^{\leq 0}$ and $\text{}^{m} D^{\geq 0}$  denote the middle-perversity ($m(2s)=-s$) $t$-structure. (For a quick introduction to $t$-structure, see \cite{B}, \cite{D2} and \cite{Sc}.) $V^{\prime}\setminus L$ and $F_{0}^{\prime}$ are both ($n+1$)-dimensional affine varieties, then Corollary 6.0.4 in \cite{Sc} shows that \begin{equation} \label{6.1}
H^{j}_{c}(F_{0}^{\prime},Rj^{\prime}_{\ast}\F)\cong H^{j}(V^{\prime}\setminus L,v_{!}\F)=0 \text{ for }j\neq n+1
\end{equation}

Consider the distinguish triangle \begin{equation} \label{6.2}
i^{\prime}_{!}i^{\prime !}\psi_{f^{\prime}}\Q\to \psi_{f^{\prime}}\Q \to Rj^{\prime}_{\ast}j^{\prime \ast}\psi_{f^{\prime}}\Q \overset {[1]}{\to}.
\end{equation}  Then we have the following compactly supported long exact sequence: 
\begin{center}
$\xymatrix{
\cdots \ar[r] & H^{2n+2-i}_{c}(F_{0}, i^{\prime !}\psi_{f^{\prime}}\Q) \ar[r]^{(1)} & H^{2n+2-i}_{c}(F_{0}^{\prime},\psi_{f^{\prime}}\Q) \ar[r] & H^{2n+2-i}_{c}(F_{0}^{\prime},Rj^{\prime}_{\ast}\F) \ar[r] & \cdots
}$
\end{center} It follows from (\ref{6.1}) that the morphism (1) is an isomorphism for $i< n$, and a surjective morphism for $i=n$. Since the four functors  in the triangle (\ref{6.2}) induced by the algebraic map $i^{\prime}$ and $j^{\prime}$ can be naturally lifted to the functors on the category of mixed Hodge modules (\cite{Sa}), this long exact sequence extends to the MHS level.

Since $L$ is transversal to $V^{\prime}\cup H^{\prime}$, Lemma 5.1.6 in \cite{Sc} shows that $i^{\prime !}=i^{\prime \ast}[-2]$, and Lemma 4.3.4 in \cite{Sc} gives that $i^{\prime \ast}\psi_{f^{\prime}}\Q\cong \psi_{f}\Q$. So, \begin{center}
$H^{2n+2-i}_{c}(F_{0}, i^{\prime !}\psi_{f^{\prime}}\Q) \cong H^{2n-i}_{c}(F_{0}, i^{\prime \ast}\psi_{f^{\prime}}\Q) \cong H^{2n-i}_{c}(F_{0}, \psi_{f}\Q) $.
\end{center}  Therefore, we have MHS isomorphisms: $H_{c}^{2n-i}(F_{0},\psi_{f}\Q)\cong H_{c}^{2n+2-i}(F^{\prime}_{0},\psi_{f^{\prime}}\Q)$ for $i<n$, and a surjective MHS morphism: $H_{c}^{n}(F_{0},\psi_{f}\Q)\to H_{c}^{n+2}(F^{\prime}_{0},\psi_{f^{\prime}}\Q)$. 
\end{remark}

\subsection{Milnor fibre $F_{\widetilde{f}}$}
The construction in Lemma \ref{l6.2} is very natural, and we can use this construction to get the following generalization of Corollary 4.9 in \cite{L3}.

\begin{cor} \label{c6.5}  Assume that the polynomial $f: \CN\rightarrow \C$ is transversal at infinity. Let $\widetilde{f}$ be the homogenization of $f$ and $F_{\widetilde{f}}$ the corresponding Milnor fibre given by $\widetilde{f}=1$. Then we have isomorphisms of $\Gamma$-modules \begin{center}
$H_{i}(\U^{c})\cong H_{i}(F_{\widetilde{f}})$  for $i\leq n$.
\end{center}   In particular, this shows that $H_{i}(\U^{c})$ is a semi-simple torsion $\Gamma$-module for $i\leq n$, which is annihilated by $t^{d}-1$.
\end{cor}
\begin{proof}  
Choose the transversal pair $(V^{\prime},H^{\prime},L)$ as defined in Lemma \ref{l6.2}. Apply the following change of coordinates: $y_{0}=x_{0}-x_{n+2}$, $y_{n+2}=x_{0}$, and $y_{i}=x_{i}$ for $1\leq i \leq n+1$. Then $H^{\prime}=\lbrace y_{0}=0 \rbrace$ and $V^{\prime}=\lbrace \widetilde{f}(y_{n+2},y_{1},\cdots,y_{n+1})=0 \rbrace$. Now $f^{\prime}=\widetilde{f}$ is already a homogeneous polynomial, so $(\U^{\prime})^{c}$ is homotopy equivalent with $F_{\widetilde{f}}$. By applying the Lefschetz hyperplane section theorem,  we obtain isomorphisms \begin{center}
 $H_{i}(\U^{c})\cong H_{i}((\U^{\prime})^{c})\cong H_{i}(F_{\widetilde{f}})$ for $0 \leq i\leq n$. 
\end{center} In fact, we showed that there exists a natural map from $\U^{c}$ to $F_{\widetilde{f}}$, which induces a homotopy $(n+1)$-equivalence. 
\end{proof}
 
The above corollary shows that, for polynomial transversal at infinity, most facts about the (torsion) Alexander modules can be reduced to the case of a homogeneous polynomial.

 We can also use this construction to give a new proof of L\^{e}'s result (\cite{Le}), which states that,  for polynomial transversal at infinity, $F_{0}$ is homotopy equivalent to a bouquet of $n$-spheres. This result is called the {\it Lefschetz Theorem on generic hyperplane complements in hypersurfaces} (see \cite{DP}). 
\begin{cor} Assume that the polynomial $f: \CN\rightarrow \C$ is transversal at infinity. Then $F_{0}$ is homotopy equivalent to a bouquet of $n$-spheres.
\end{cor}
\begin{proof} By Corollary \ref{c6.5}, can assume that $f^{\prime}=\widetilde{f}$ is homogeneous.
Lefschetz hyperplane section theorem gives that the inclusion $F_{0}\hookrightarrow F_{0}^{\prime}$ is a homotopy $n$-equivalence.  Note that the affine variety $F_{0}^{\prime}= \lbrace \widetilde{f}=0 \rbrace$ is contractible. Then $F_{0}$ is $(n-1)$-connected. $F_{0}$ has the homotopy type of a finite $n$-dimensional CW complex, then Theorem 6.5 in \cite{M} shows that $F_{0}$ is homotopy equivalent to a bouquet of $n$-spheres.
\end{proof}

  Denote $h=f_{d}$.  Since $V$ intersects $H$ transversally, the hyperplane defined by $\lbrace x_{0}=0 \rbrace$ in $\C^{n+2}$ is `generic' for $\widetilde{f}$ (for the statement of `generic', see \cite{DP}, p.483). Then the Milnor fibre of $\widetilde{f}$ is obtained (up to homotopy) from the Milnor fibre of $h$  by attaching $(n+1)$-cells (\cite{DP}, Proposition 9). So this natural inclusion map induces isomophisms $ H_{i}(F_{h}) \cong H_{i}(F_{\widetilde{f}})$ for $i\leq n-1$, and a surjective map $H_{n}(F_{h})\rightarrow H_{n}(F_{\widetilde{f}})$ for $i=n$. Corollary \ref{c6.5} shows that this is compatible with Proposition 4.9 in  \cite{LM2}: $H_{i}(\U^{c})\cong H_{i}(F_{h})$ for $i\leq n-1$, and $H_{n}(\U^{c})$ is a quotient $\Q$-vector space of $H_{n}(F_{h})$.

\begin{remark} As an application of Corollary \ref{c6.5}, we can give obstructions on the eigenvalues
of monodromy operators associated to the Milnor fibre $F_{\widetilde{f}}$, just like Maxim did in section 5 of \cite{LM2}. In particular, the general divisibility results (Theorem \ref{t1.3} in this paper and Theorem 4.2 in \cite{LM2}) hold for $H_{i}(F_{\widetilde{f}})$, where $i\leq n$.
\end{remark}

\begin{example} Assume that $Sing(F_{0})$ itself is a stratum in the stratification, denoted by $\ST$, and $\ST$ is $(s-1)$-connected with $s>1$, where $s$ is the complex dimension of $\ST$. This case has been studied in  section 5.2.2. Note that $H_{i}(F_{h})\cong H_{i}(\U^{c})$ for $i<n$ (\cite{LM2}, Proposition 4.9), so the non-vanishing $H_{i}(F_{h})$ are only in degrees $i=0,n-s,n$.  By Corollary \ref{c6.5}, the non-vanishing $H_{i}(F_{\widetilde{f}})$ are only in degrees $i=0,n-s,n,n+1$. It is interesting to compare this result with Theorem 1 in \cite{LM}. 
\end{example}

\subsection{Mixed Hodge structures on Alexander modules} In this subsection, we show that the above mentioned MHS coincide (including the MHS defined by Dimca and Libgober). The key point is the construction in Corollary \ref{c6.5}, which shows that we can choose $f^{\prime}=\widetilde{f}$ homogeneous.

Assume $(V^{\prime},H^{\prime},L)$ is a transversal pair for $(V,H)$, and define $F_{0}^{\prime}$, $f^{\prime}$, $\U^{\prime}$, $(\U^{\prime})^{c}$ and $F^{\prime}$ (the generic fibre of $f^{\prime}$) associated to $V^{\prime}$ and $H^{\prime}$.  We have the following isomorhisms:  \begin{center}
$H_{i}(\U^{c})\cong H_{i}(F_{\widetilde{f}})\cong H_{i}(F^{\prime})\cong H^{2n+2-i}_{c}(F_{0}^{\prime},\psi_{f^{\prime}}\Q)$ for $i\leq n$.
\end{center}
This shows that we have three ways to define MHS on $H_{i}(\U^{c})$ for $i\leq n$:

(a) Use the nearby cycles to define MHS. This is what we did in the proof of Theorem \ref{t1.5}.

(b) the limit MHS on the Milnor fibre $H_{i}(F_{\widetilde{f}})$.   
 
(c) the limit MHS on the generic fibre $H_{i}(F^{\prime})$. \medskip 

 In fact, these three alternatives define the same MHS. Lemma \ref{l6.3} shows that the MHS defined by these three methods do not depend on the choice of the transversal pair. So we can choose the transveral pair $(V^{\prime},H^{\prime},L)$ as defined in Lemma \ref{l6.2}. Apply the change of coordinates as we did in Corollary \ref{c6.5}, so that $f^{\prime}=\widetilde{f}$ is homogeneous. Therefore, in this case, the generic fibre $F^{\prime}$ is same as the Milnor fibre $F_{\widetilde{f}}$, and  $ H_{i}(F_{\widetilde{f}})$ carries the same MHS as $H^{2n+2-i}_{c}(F_{0}^{\prime},\psi_{f^{\prime}}\Q)$ for all $i$ (see \cite{BS}). In particular, this MHS is canonical: compatible with the action of $\Q[t,t^{-1}]$, i.e., $t: H_{i}(F_{\widetilde{f}})\rightarrow H_{i}(F_{\widetilde{f}})$ is a MHS morphism.
 
 \medskip 

Consider the following three natural inclusions:

(a)  $F_{0} \hookrightarrow F_{0}^{\prime}$ induces the natural MHS morphism $H^{2n-i}_{c}(F_{0},\psi_{f}\Q) \rightarrow H^{2n+2-i}_{c}(F_{0}^{\prime},\psi_{f^{\prime}}\Q)$, which is an isomorphism for $i<n$ and a surjective morphism for $i=n$ (see Remark \ref{r6.4}).   

(b) $F_{h}\hookrightarrow F_{\widetilde{f}}$ is a homotopy $n$-equivalence with compatible monodromy. Then it induces a natural MHS morphism $H_{i}(F_{h})\rightarrow H_{i}(F_{\widetilde{f}})$, which is an isomorphism for $i<n$ and a surjective morphism for $i=n$. 
 
(c) Choose $c\in \C$ such that  $F=f^{-1}(c)$ and $F^{\prime}=(f^{\prime})^{-1}(c)$ are both generic fibres.  There is a commutative diagram as follows: \begin{eqnarray} \label{6.3}
\xymatrix{
F \ar[r]^{i_{c}} \ar[d]_{u^{\prime}}  & \U^{c} \ar[d]^{u}  \\
F^{\prime} \ar[r]^{i^{\prime}_{c}}    &  (\U^{\prime})^{c}                 
}
\end{eqnarray}
 Here $i_{c}$ and $i_{c}^{\prime}$ are the natural inclusions in \cite{DL}, where $i_{c}$ and $i_{c}^{\prime}$ induce a homotopy $n$-equivalence and a homotopy $(n+1)$-equivalence, respectively; $\U^{c}\hookrightarrow (\U^{\prime})^{c}$  is a homotopy $(n+1)$-equivalence by the Lefschetz hyperplane section theorem. Then the natural inclusion $F\hookrightarrow F^{\prime}$ is a homotopy $n$-equivalence. Therefore $F\hookrightarrow F^{\prime}$ induces a natural MHS morphism $H_{i}(F) \rightarrow H_{i}(F^{\prime})$, which is an isomorphism for $i<n$ and a surective morphism for $i=n$.

\medskip
The above discussion finishes the proof of Proposition \ref{p1.6}. In particular, when $\mu=0$, Proposition \ref{p4.2} shows that these 3 surjetive morphisms become isomorphisms. So, there exist MHS isomorphisms \begin{center}
$H_{i}(\U^{c})\cong H^{2n-i}_{c}(F_{0},\psi_{f}\Q)\cong H_{i}(F)\cong H_{i}(F_{h})$  for all $i$.
\end{center}  This shows that, when $\mu=0$, $f$ behaves, both topologically and algebraically (e.g., in terms of the variation of MHS on the cohomology of its smooth fibers), like a homogeneous polynomial.

Dually, for $i \leq n$, the Alexander module $H^{i}(\U^{c})$ has a canonical mixed Hodge structure.

Dimca and Libogber define MHS on $H^{i}(\U^{c})$ by the sub-MHS of $H^{i}(F)$ induced by $i_{c}^{\ast}$ for $i\leq n$. The diagram (\ref{6.3}) shows that there exist MHS isomorphisms \begin{center}
$i_{c}^{\ast}H^{i}(\U^{c})\cong i_{c}^{\ast} u^{\ast} H^{i}((\U^{\prime})^{c})\cong u^{\prime \ast} i_{c}^{\prime \ast}H^{i}((\U^{\prime})^{c})\cong u^{\prime \ast}(H^{i}(F^{\prime}))$ for $i\leq n$.
\end{center}  Therefore we indeed recover the MHS defined by Dimca and Libgober. As a corollary, this also shows that the MHS on $H^{i}(F_{c})$ is independent of choice of $c\in \C\setminus B_{f} $ for $i<n$, where $F_{c}=f^{-1}(c) $ is the generic fibre (see \cite{DL}, Corollary 1.6).


\end{document}